\documentclass{article}
\usepackage[T1]{fontenc}
\usepackage[utf8]{inputenc}
\usepackage{float}
\usepackage[export]{adjustbox}
\usepackage[margin=1.5in]{geometry}
\usepackage{graphicx}
\graphicspath{ {./images/} }
\usepackage{amsmath,amssymb,amsthm}
\numberwithin{equation}{section}
\usepackage[dvipsnames]{xcolor}
\usepackage{enumitem,amssymb}
\usepackage{hyperref}
\usepackage{todonotes}
\usepackage{url}\usepackage{caption}
\usepackage{subcaption}
\usepackage[nameinlink, capitalise, noabbrev]{cleveref}
\usepackage[authoryear,sort,round, numbers]{natbib}
\usepackage{hyphenat}
\makeatletter
\newcommand{\nosemic}{\renewcommand{\@endalgocfline}{\relax}}
\newcommand{\dosemic}{\renewcommand{\@endalgocfline}{\algocf@endline}}

\let\oldnl\nl
\newcommand{\nonl}{\renewcommand{\nl}{\let\nl\oldnl}}
\makeatother
\usetikzlibrary{3d,calc}
\newtheorem{theorem}{Theorem}[section]
\newtheorem{lemma}[theorem]{Lemma}
\newtheorem{proposition}[theorem]{Proposition}
\newtheorem{assumption}[theorem]{Assumption}

\newtheorem{corollary}[theorem]{Corollary}
\numberwithin{equation}{section}
\newcommand{\N}{\mathbb{N}}

\newcommand{\R}{\mathbb{R}}

\DeclareMathOperator{\spann}{span}

\DeclareMathOperator{\argmin}{arg\,min}

\usepackage{amsmath}

\newcommand\norma[1]{\left|#1\right|}

\usepackage{mathtools}
\Crefname{algocf}{Algorithm}{Algorithms}

\definecolor{dgreen}{RGB}{0,100,0}

\title{Accuracy Boost in Ensemble Kalman Inversion via Ensemble Control Strategies}
\makeatletter
\author{Ruben Harris and Claudia Schillings \footnote{Fachbereich Mathematik und Informatik, Freie Universit\"at Berlin, Arnimallee 6, 14195 Berlin, Germany, $\{$r.harris, c.schillings$\}$@fu-berlin.de} }
\makeatother
\date{}

\usepackage{cleveref}
\usepackage{algpseudocode}
\usepackage{algorithm}
\newenvironment{varalgorithm}[1]
  {\algorithm\renewcommand{\thealgorithm}{#1}}
  {\endalgorithm}
\begin{document}

\maketitle
\begin{abstract}   
The Ensemble Kalman Inversion (EKI) method is widely used for solving inverse problems, leveraging ensemble-based techniques to iteratively refine parameter estimates. Despite its versatility, the accuracy of EKI is constrained by the subspace spanned by the initial ensemble, which may poorly represent the solution in cases of limited prior knowledge. This work addresses these limitations by optimising the subspace in which EKI operates, improving accuracy and computational efficiency. We derive a theoretical framework for constructing optimal subspaces in linear settings and extend these insights to nonlinear cases. A novel greedy strategy for selecting initial ensemble members is proposed, incorporating prior, data, and model information to enhance performance. Numerical experiments on both linear and nonlinear problems demonstrate the effectiveness of the approach, offering a significant advancement in the accuracy and scalability of EKI for high-dimensional and ill-posed problems.
\end{abstract}

\section{Introduction}
Mathematical models are indispensable for understanding a wide array of systems of scientific interest. These models facilitate analysis and prediction of system behaviours, but their accuracy relies on calibration to observed data — a process referred to as `inversion'. Inversion is the cornerstone of fields such as weather forecasting, medical imaging and machine learning that employ various techniques categorised as either variational/optimisation-based or Bayesian/statistical.

One notable inversion method is Ensemble Kalman Inversion (EKI), which integrates elements from both variational and Bayesian approaches. Introduced by \cite{iglesias2013ensemble}, EKI iteratively employs an Ensemble Kalman-Bucy Filter to tackle inverse problems. In linear scenarios, this takes the form of a preconditioned gradient flow equation (see \cite{schillings2017analysis}). 

The structure of the EKI update step results in the so-called subspace property \cite{iglesias2013ensemble}, whereby the EKI estimate remains confined to the subspace spanned by the initial ensemble. While this property acts as a natural regulariser, it can also lead to poor approximations if the subspace does not adequately capture the true solution. Traditionally, the initial ensemble is derived from prior knowledge and so, in the case of limited or uninformative priors, the chosen initial ensemble may result in a subspace poorly aligned with the true solution.

This work addresses these limitations by optimising the subspace within the EKI framework to improve both computational efficiency and solution accuracy. By constructing a subspace that takes into account the forward operator and the data set in addition to just the priors, we aim to maximise accuracy within a fixed computational budget (in terms of particle count). This refined approach leverages the structure of the inverse problem to make EKI more scalable and accuracte, especially when dealing with limited or inaccurate prior information.

\subsection{Literature overview}\label{sub:existingmethods}
Since it was first developed by \cite{Evensen2003}, the Ensemble Kalman Filter (EnKF) has become widely used in both data assimilation and inverse problems. Its popularity is largely due to its straightforward implementation and effectiveness even with smaller ensemble sizes, as shown in various studies (see \cite{Bergemann2009, Bergemann2010, Iglesias2014, Iglesias2016, iglesias2013ensemble, Li2009}). This feature makes EnKF particularly well-suited to high-dimensional problems, where other inversion methods may face computational challenges. EKI was first introduced by \cite{iglesias2013ensemble} as an extension of EnKF to inverse problems. EKI sets up artificial dynamics in which the target data acts as the observed signal, which is repeatedly integrated into the model estimates. A key advantage of EKI over traditional inversion methods is its derivative-free formulation, which allows it to function as a computationally efficient black-box method.

Research on EKI has increasingly focused on understanding convergence and stability. Much of this analysis was has been based on studying the continuous-time limit of EKI. Stability studies have been initiated in \cite{Tong2015, Tong2016} and convergence analysis to and in the continuous time limit can be found in \cite{Bloemker2019, Bloemker2021, bungert2021complete, schillings2017analysis, Schillingsnoisy}. These studies show that convergence in parameter space often requires regularisation; Tikhonov regularisation has been extensively examined in this context (see \cite{Iglesias_2016, Tong2020, Scherzer}). Recently, adaptive regularisation strategies were introduced to increase the robustness of stochastic EKI (see \cite{Iglesias_2021, Weissmann2022}). The mean-field limit of EKI has also been analysed, providing valuable insights into ensemble dynamics as the ensemble size grows large (see \cite{Stuart2022, Ding2020}).

To gain a foundational understanding of EKI, researchers have frequently studied its behaviour under linear forward operators.  \cite{iglesias2013ensemble} demonstrated that - in the linear, Gaussian setting - EKI converges to the maximum a posteriori (MAP) estimate in the mean-field limit (or, equally, minimises a Tikhonov-regularised objective). In general, however, the ensemble is not an approximation of the posterior in the Bayesian setting (see \cite{ErnstEtAl2015}). Further analysis by \cite{schillings2017analysis} 
showed that in the linear, noiseless regime, EKI experiences ensemble collapse, whereby particle positions converge to an estimate of the MAP within a specific subspace. Connections to other optimisation methods have been explored in \cite{iterKalman,Weissmann_2022gr}. The evolution of the EKI ensemble was further analysed by \cite{bungert2021complete}, who provided explicit expressions for the particle trajectories in the linear, noiseless regime. In the discrete setting, an analysis of the particle trajectories can be found in \cite{qian2024fundamentalsubspacesensemblekalman}.

In practical applications, the choice of initial ensemble is particularly important, especially since the number of particles in the ensemble is often far smaller than the dimension of the observation or state space. \cite{iglesias2013ensemble} showed that particles remain in the subspace spanned by the initial ensemble, thus regularising the inverse problem. This subspace is therefore a critical design parameter in EKI. When prior knowledge is represented by a Gaussian distribution, two strategies have been suggested for selecting initial ensemble members: random sampling from the Gaussian prior or utilising the terms in the Karhunen-Loève expansion that correspond to the largest eigenvalues. In the linear setting, \cite{ghattas2022non} provided theoretical support for the random sampling approach, showing that this technique approximates the posterior mean and covariance with error bounds that decay according to particle count and the `effective dimension' of the prior covariance. 
Certain techniques, such as covariance inflation and localisation, alter the empirical covariance in standard EKI. As a result, updates may leave the span of the initial ensemble, breaking the subspace property and its regularising effect (see \cite{Evensen2022}). 

This paper is concerned with optimising the subspace within which the ensemble evolves and the choice of initial ensemble within this subspace. This refinement of the initial subspace promises to improve the accuracy of EKI and its scalability to high-dimensional inverse problems. The continued development of subspace optimisation, alongside regularisation techniques, holds significant potential for improving EKI's performance in computationally challenging settings. The idea of subspace optimisation in EKI is conceptually related to the likelihood-informed subspace method, which reduces the dimensionality of high-dimensional probability distributions in Bayesian inference by  using the leading eigenvectors of a Gram matrix derived from the gradient of the log-likelihood function to identify a low-dimensional subspace where the target distribution diverges most from a reference distribution. We refer to \cite{LIS} for a detailed explanation of the likelihood informed subspace method.

\subsection{Main ideas and contributions of the paper}
In this work, we focus on improving the accuracy of EKI with small ensembles by optimising the subspace in which the algorithm operates. We aim to improve on the subspace arising from the standard ensemble initialisation, which might lead to low accuracy in the case of poor prior knowledge. By optimising the EKI subspace, we aim to maximise accuracy within a given computational budget (i.e. with a given number of particles). Our contributions are as follows:
\begin{itemize}
\item We derive closed-form expressions for the linear, noise-free particle dynamics. Though these new expressions are equivalent to those already derived by \cite{bungert2021complete}, their particular form facilitates the analysis of the choice of ensemble subspace.
\item We generalise a standard approach for generating the initial ensemble by (1) expanding the choice of subspaces in which to situate the initial ensemble (2) allowing the initial ensemble to consist of any set of linearly independent vectors in a chosen subspace. Using our analysis of the dynamics of the particles in the linear, noiseless regime, we evaluate the `goodness' of a given choice of subspace and linear combination.
\item We propose a scalable, greedy strategy to select good subspaces. We characterise families of `optimal' linear combinations given a choice of subspace and propose concrete methods for generating such linear combinations. These selection strategies incorporate prior, data and model knowledge.
\item Based on insights from the linear, noiseless case, we propose a strategy for adaptively resampling the EKI ensemble that can be applied to general forward operators (i.e. not just linear forward operators).
\item We numerically demonstrate the effectiveness of our approach by applying it to both linear and nonlinear example problems.
\end{itemize}

This paper is organised as follows. In Sections \ref{sec:IP} and \ref{sec:EKI}, we introduce the problem setting and EKI methodology. 
Section \ref{sec:subspace_optim} covers the analysis of the linear, noiseless particle dynamics, the characterisation of optimal initial ensembles in this setting, and strategies for constructing good ensembles in both the linear and nonlinear setting. Numerical experiments are presented in Section \ref{sec:NumExp} and conclusions are drawn in Section \ref{sec:concl}. Auxiliary results and proofs of the main theorems can be found in the appendices.
\subsection{Notation}
For any $n  \in \mathbb{N}^{+}$, we consider $\mathbb{R}^n$, equipped with either the standard inner product $\langle \cdot, \cdot \rangle$ and its norm $| \cdot |$ or the weighted inner product $\langle \cdot, \cdot \rangle_{\Gamma} := \langle \cdot, \Gamma^{-1} \cdot \rangle $ and its norm $| \cdot |_\Gamma$, where $\Gamma \in \mathbb{R}^{n \times n}$ is symmetric and positive definite. Further, $(\Omega, \mathcal{A}, \mathbb{P})$ denotes a probability space, $X$ denotes a separable Hilbert space and $\mathcal{B}X:=\mathcal{B}(X,|\cdot|)$ denotes the Borel $\sigma$-algebra on $X$ (or $\mathcal{B}X:=\mathcal{B}(X,|\cdot|_\Gamma)$). Introducing an additional space $\mathbb{R}^m$ with $m\in\mathbb{N}^{+}$, we denote the tensor product of vectors $x \in \mathbb{R}^n$ and $y \in \mathbb{R}^m$ by $x \otimes y := x y^\top$.
\section{Inverse Problem}\label{sec:IP}
In the inverse problem setting, our goal is to identify the unknown parameters \( u \in X \) from noisy observations:
\begin{equation} \label{eq_IP}
y = G(u) + \eta,
\end{equation}

where \( \eta \in Y \) represents additive observational noise, and \( G: X \rightarrow Y \) is the operator that solves the underlying forward problem, with $X$ being a separable Hilbert space and $Y=\mathbb{R}^m$ for some $m \in \mathbb{N}^{+}$.

Within the Bayesian framework, \( u \) and \( \eta \) are modelled as independent random variables - \( u: \Omega \rightarrow X \) and \( \eta: \Omega \rightarrow Y \) - with the assumption that \( u \perp \eta \). Assuming the noise \( \eta \) follows a Gaussian distribution \( \eta \sim \mathrm{N}(0, \Gamma) \) and the operator \( G \) is continuous, the posterior distribution \( \mu^y \) is given by

\begin{equation} \label{eq:bip}
\mathrm{d} \mu^y(u) = \frac{1}{Z} \exp\left(-\frac{1}{2} \lvert y - G(u) \rvert_{\Gamma}^2\right) \mathrm{d} \mu_0(u),
\end{equation}

where \( \mu_0 \) denotes the prior distribution on the unknown parameters, \( l(u) := \frac{1}{2} \lvert y - G(u) \rvert_{\Gamma}^2 \) represents the likelihood and \( Z = \mathbb{E}_{\mu_0} \left[ \exp\left(-\frac{1}{2} \lvert y - G(u) \rvert^2_{\Gamma} \right)\right] \) is the normalisation constant (which is assumed to be positive).

The primary focus of our study is the computation of the Maximum a Posteriori (MAP) estimate, which provides a point estimate for the unknown parameter. In a finite\\-dimensional context, this involves minimising the negative log density. Given a Gaussian prior on the unknown parameters \( u \sim \mathcal{N}(\mu, R) \), the objective function that we seek to minimise is given by

\begin{equation} \label{eq:pot_reg}
\Phi(u) := \frac{1}{2} \lvert y - G(u) \rvert_{\Gamma}^2 + \frac{1}{2} \lvert u -\mu\rvert^2_{R},
\end{equation}

which means the MAP estimate is the point that maximises the posterior density. For a generalisation of MAP points to the infinite-dimensional setting, we refer to \cite{Dashti_2013, Klebanov_2023}.

We note that, under the transformations $\tilde{u} := u - \mu$, $\tilde{y} := \Gamma^{-\frac12}y$, $\tilde{G}(u) := \Gamma^{-\frac12}G(u + \mu)$, we can equivalently minimise the following objective function: 

\begin{equation*}
\frac{1}{2} \lvert \tilde{y} - \tilde{G}(\tilde{u}) \rvert^2 + \frac{1}{2} \lvert \tilde{u}\rvert^2_{R},
\end{equation*}

As such, we assume, without loss of generality, that $\Gamma = I$ and $\mu = 0$ in our analysis in Section \ref{sec:subspace_optim}.

\section{Ensemble Kalman Inversion}\label{sec:EKI}

Our objective is to address the inverse problem specified in Equation \eqref{eq_IP} using the Ensemble Kalman Inversion (EKI) framework, with a particular focus on the continuous-time version of EKI that is described in \cite{schillings2017analysis}. We begin by outlining the derivation of EKI in discrete time and then discuss the transition to the continuous-time limit. Specifically, we employ a homotopy approach that approximates the posterior distribution $\mu$ defined in \eqref{eq:bip} through a sequence of probability measures $\mu_n$ for $n = 1, ..., N$:

\begin{equation*}
\mu_{n}(du) \propto \exp\bigl(-nhl(u)\bigr)\mu_0(du), \quad h=N^{-1},
\end{equation*}

where the target measure $\mu_N = \mu$ represents the posterior distribution of u given $y$. The update from $\mu_n$ to $\mu_{n+1}$ is defined by

\begin{equation*}
\mu_{n+1}(du) = \frac{1}{Z_n}\exp\bigl(-hl(u)\bigr)\mu_n(du),
\end{equation*}

with the normalisation constant $Z_n$ computed as

\[
Z_n = \int \exp(-hl(u))\mu_n(du).
\]

These intermediate measures can be interpreted as the repeated use of the observational data in a filtering context, where the observational noise is amplified by a factor of $1/h$. The measures are then approximated by an ensemble of J particles for some \( J \in \mathbb{N}^{+} \). The initial ensemble is defined as \( u_0 := \{u_0^{(i)}\}_{j = 1}^{J} \in X^{J} \).

Applying the Ensemble Kalman Filter (EnKF) to this problem results in the following iterative update:

\begin{equation} \label{eq:EnKFilter}
u_{n+1}^{(i)} = u_{n}^{(i)} + C^{up}(u_n)\left(C^{pp}(u_n) + h^{-1}\Gamma\right)^{-1}\left(y_{n+1}^{(i)} + \zeta_{n+1}^{(i)} - {G}(u_n^{(i)})\right), \quad j=1,\ldots,J,
\end{equation}

where $\zeta_{n+1}^{(i)} \sim \mathcal{N}(0, I_{k})$ i.i.d. The operators \( C^{pp} \) and \( C^{up} \) are defined as

\begin{align*}
C^{pp}(u) &= \frac{1}{J}\sum_{j=1}^J\left({G}(u^{(j)})-\overline{{G}}\right) \left({G}(u^{(j)})-\overline{{G}}\right)^{T}, \\
C^{up}(u)(\cdot) &= \frac{1}{J}\sum_{j=1}^J \langle{G}(u^{(j)})-\overline{{G}},\  \cdot \ \rangle \left(u^{(j)}-\overline{u}\right), \\
\overline{u} &= \frac{1}{J}\sum_{j=1}^J u^{(j)}, \quad \overline{{G}} = \frac{1}{J}\sum_{j=1}^J {G}(u^{(j)}).
\end{align*}

While this approach offers a practical solution to Bayesian inverse problems, it does not guarantee convergence to \( \mu_n \) as \( J \rightarrow \infty \) due to inherent approximation errors (see \cite{ErnstEtAl2015}). Our focus is on understanding the behaviour of the algorithm with a fixed number of particles, particularly when applied as a derivative-free optimiser to the minimisation problem \eqref{eq:pot_reg}.

In discussing EKI, it is important to highlight the invariant subspace property noted in \cite{iglesias2013ensemble}:

\begin{lemma}
\label{lem:invariant_subspace}

Let ${\mathcal S}$ denote the affine space $\mathcal S=\overline{u}_0+\spann \{u_0^{(i)}-\overline{u}_0, \quad i=1,\ldots, J\}$. Then $u_n^{(i)} \in {\cal S}$ for all $(n,i) \in \N^+ \times \{ 1, \ldots, J\}$.
\end{lemma}
We note that the subspace property allows us to reformulate the possibly infinite dimensional problem as a finite dimensional one by restricting the iterations to the finite dimensional subspace spanned by the initial ensemble. In our analysis, we assume that the parameter space $X$ is finite dimensional, i.e. $X=\mathbb R^n$ for some $n\in \mathbb N^{+}$.

\subsection{Continuous-Time Limit}

We investigate the continuous-time limit of EKI as the parameter \( h \) in the incremental Bayesian formulation approaches zero. For a detailed and rigorous derivation of these limits, we refer to \cite{bloemker2018strongly, Bloemker2021, lange2021derivation}. The EKI update step \eqref{eq:EnKFilter} can be interpreted as a time-stepping scheme. As \( h \to 0 \), this formulation converges to a tamed discretisation of the following coupled stochastic differential equations:

\begin{equation}\label{eq:sde}
\frac{d u^{(i)}}{dt} = C^{up}(u)\Gamma^{-1}\left(y - {G}(u^{(i)})\right) + C^{up}(u^{(i)})\Gamma^{-1}\sqrt{\Gamma} \frac{dW^{(i)}}{dt},
\end{equation}

and more explicitly:

\begin{equation}\label{eq:sde2}
\frac{d u^{(i)}}{dt} = \frac{1}{J}\sum_{k=1}^J \bigl\langle {G}(u^{(k)}) - \bar{{G}}, y - {G}(u^{(i)}) + \sqrt{\Gamma} \frac{dW^{(i)}}{dt}\bigr\rangle_{\Gamma} (u^{(k)} - \bar{u}),
\end{equation}

where the independent cylindrical Brownian motions \( W^{(i)} \) on \( X \) result from the perturbed observations in the original formulation. From \eqref{eq:sde2}, it is clear that the subspace property remains intact in continuous time. By setting the observation perturbations to zero in each iteration, the continuous-time limit simplifies to a deterministic system:

\begin{equation*}\label{eq:ode}
\frac{d u^{(i)}}{dt} = C^{up}(u)\Gamma^{-1}\left(y - {G}(u^{(i)})\right).
\end{equation*}

This continuous-time formulation serves as the foundation for further analysis regarding the optimal choice of the initial ensemble.

As indicated in Section \ref{sec:IP}, we assume, without loss of generality, that $\Gamma = I$ in the following analysis.

\section{Adaptive Selection of the Initial Ensemble}\label{sec:subspace_optim}

Due to the subspace property (Lemma \ref{lem:invariant_subspace}), the EKI approximations are constrained to the subspace spanned by the initial ensemble. This section, therefore, focuses on strategies for optimising the selection of this subspace.

We begin by analysing the case where the forward operator, $A: X \to Y$, is linear and later extend our analysis to the nonlinear case.

In the linear case, the dynamics of the particle system $\{u^{(i)}(t)\}_{i = 1}^{J} \in \R^{n}$ are governed by the following set of coupled ordinary differential equations (ODEs):
\begin{align}\label{eq:particle_ODE}
\frac{d u^{(i)}(t)}{dt} &= -C_{t}A^{T}(Au^{(i)}(t) - y), \nonumber\\
C_{t} &= \frac{1}{J}\sum_{j=1}^{J}(u^{(j)} - \bar{u})(u^{(j)} - \bar{u})^{T}.
\end{align}

where $C_{t} \in \mathbb{R}^{n \times n}$ is the empirical covariance matrix of the particle system. It is noteworthy that, in the linear case, EKI can be interpreted as a preconditioned gradient flow (see \cite{schillings2017analysis}).

By defining $\mathcal{E}_{t} := [u^{(1)}(t) - \bar{u}(t), \ldots , u^{(J)}(t) - \bar{u}(t)] \in \mathbb{R}^{n \times J}$, we express the covariance matrix as $C_{t} = \frac{1}{J}\mathcal{E}_{t}\mathcal{E}_{t}^{T} \in \mathbb{R}^{n \times n}$. If we further define $U_{t} = [u^{(1)}(t), \ldots , u^{(J)}(t)] \in \mathbb{R}^{n \times J}$ and $Y = [y, \ldots , y] \in \mathbb{R}^{m \times J}$, the system can be rewritten in matrix form as

\begin{equation*}\label{matrix_ODE}
\frac{d}{dt} U_{t} = -C_{t}A^{T}(AU_{t} - Y) = -\frac{1}{J}\mathcal{E}_{t}(A\mathcal{E}_{t})^{T}(AU_{t} - Y).
\end{equation*}

\cite{bungert2021complete} provide a comprehensive analysis of this system of ODEs, culminating in explicit expressions for the particle trajectories. Building on this foundation, we derive alternative expressions for the particle trajectories that we believe offer a different perspective on the system's dynamics. The key difference between these perspectives is that the solutions presented in \cite{bungert2021complete} rely on the diagonal decomposition $C_{0}A^{T}A = SDS^{-1}$, whereas our solutions rely on the diagonal decomposition of the symmetric matrix $AC_{0}A^{T} =: \mathcal{U}\Sigma\mathcal{U}^{T}$ or the compact singular value decomposition $A\mathcal{E}_{0} =: \sqrt{J}\mathcal{U}\Sigma^{\frac{1}{2}}\mathcal{V}^{T}$.

\subsection{Analysis of Particle Dynamics}
In this Subsection, we analyse the dynamics of an ensembles moving under deterministic EKI with a linear forward operator, i.e. particles moving according to ODEs \eqref{eq:particle_ODE}. By focusing on the linear forward operator without noise, we derive explicit solution to ODEs \eqref{eq:particle_ODE} and characterise the behaviour of particles as they evolve over time. In particular, we see how 
the choice of the initial ensemble affects the accuracy of the EKI in the long-term. This analysis is also the starting point for understanding EKI methods in the nonlinear setting. We begin by stating the necessary assumptions and definitions used throughout our analysis.
\begin{assumption}\phantom{~\\}
\begin{enumerate}[label={\rm (A2.\arabic*)},align=right,leftmargin=2.5\parindent]
\item The forward operator $A: X \to Y$ is linear , i.e. we identify the operator with a matrix $A \in \mathbb{R}^{k \times n}$.\label{eq:A1}
\item We define the initial ensemble $U_{0} := [u^{(1)}(0), \ldots, u^{(J)}(0)] \in \mathbb{R}^{n \times J}$, the deviations $\mathcal{E}_{0} := [u^{(1)}(0) - \bar{u}(0), \ldots, u^{(J)}(0) - \bar{u}(0)] \in \mathbb{R}^{n \times J}$ and the empirical covariance $C_{0} := \frac{1}{J} \mathcal{E}_{0}\mathcal{E}_{0}^{T} = \frac{1}{J}\sum_{j = 1}^{J}(u^{(j)}(0) - \bar{u}(0))(u^{(j)}(0) - \bar{u}(0))^{T} \in \mathbb{R}^{n \times n}$.\label{eq:A2}
\item The matrix $AC_{0}A^{T}$ is of rank $r \le k$, $AC_{0}A^{T} = \mathcal{U} \Sigma \mathcal{U}^{T}$ is an eigen-decomposition and $A\mathcal{E}_{0} = \sqrt{J} \mathcal{U}\Sigma^{\frac{1}{2}}\mathcal{V}^{T}$ is a corresponding compact singular value decomposition ($\mathcal{U} \in \mathbb{R}^{m \times r}$, $\Sigma \in \mathbb{R}^{r \times r}$, $\mathcal{V} \in \mathbb{R}^{J \times r}$, $\Sigma \succ 0$).\label{eq:A3}
\end{enumerate}
\end{assumption}

\begin{theorem}\label{thm:dynamics_EKI}
Under Assumptions {\rm \ref{eq:A1}}--{\rm \ref{eq:A3}}, the initial value problem for deterministic EKI
\begin{align}\label{eq:IVPEKI}
\begin{split}
\frac{du^{(i)}(t)}{dt} &= -C_{t}A^{T}(Au^{(i)}(t) - y), \quad t > 0,\\
u^{(i)}(0) &= u_0^{(i)},
\end{split}
\end{align}
with $C_{t} := \frac{1}{J}\sum_{j = 1}^{J}(u^{(j)}(t) - \bar{u}(t))(u^{(j)}(t) - \bar{u}(t))^{T}$ and $ \bar{u}(t) = \frac{1}{J}\sum_{j = 1}^{J}u^{(j)}(t)$ for $t \ge 0$
has the unique solution given by
\begin{align*}
u^{(i)}(t) &= u^{(i)}(0) + (AC_{0})^{T} \mathcal{U}\Sigma^{-1}((I_{r} + 2 \Sigma t)^{-\frac{1}{2}} - I_{r})\mathcal{U}^{T}(Au^{(i)}(0) - y)\\
&= u^{(i)}(0) + \frac{1}{\sqrt{J}}U_{0}\mathcal{V}\Sigma^{-\frac{1}{2}}((I_{r} + 2 \Sigma t)^{-\frac{1}{2}} - I_{r})\mathcal{U}^{T}(Au^{(i)}(0) - y), \quad t \ge 0.
\end{align*}
\end{theorem}

\begin{proof}
From the ODE for $u^{(i)}(t)$, we have
\begin{equation*}
\frac{d}{dt}C_{t} = -2C_{t}A^{T}AC_{t}, \quad t > 0.
\end{equation*}
The first equality follows directly from Theorem \ref{thm:general_dynamics_thm} with $\alpha = 2$. The second equality follows from Lemma \ref{lemma:EV_UV}, which shows that $(AC_{0})^{T} \mathcal{U}\Sigma^{-1} = \frac{1}{\sqrt{J}}\mathcal{E}_{0}\mathcal{V}\Sigma^{-\frac{1}{2}} = \frac{1}{\sqrt{J}}U_{0}\mathcal{V}\Sigma^{-\frac{1}{2}}$.
\end{proof}

Theorem \ref{thm:dynamics_EKI} allows us to quantify the residual in the image space as follows:

\begin{corollary}\label{terminal_residual}
Under Assumptions {\rm \ref{eq:A1}}--{\rm \ref{eq:A3}}, the residuals in the image space are given by
\begin{equation*}
Au^{(i)}(t) - y = (I_{m} - \mathcal{U}\mathcal{U}^{T} + \mathcal{U}(I_{r} + 2\Sigma t)^{-\frac{1}{2}}\mathcal{U}^{T})(Au^{(i)}(0) - y),
\end{equation*}
and therefore
\begin{equation*}\label{eq:terminal projection}
\lim _{t\to \infty}  Au^{(i)}(t) - y = \text{Proj}_{\text{\mbox{ker}}(AC_{0}A^{T})}(Au^{(i)}(0) - y).
\end{equation*}
\end{corollary}

The next result charecterises the deviations of each particle from the empirical mean.

\begin{corollary}
Under Assumptions {\rm \ref{eq:A1}}--{\rm \ref{eq:A3}}, the deviations of the particles from their mean are given by
\begin{align*}
\mathcal{E}_{t} &= \mathcal{E}_{0} + (AC_{0})^{T}\mathcal{U}\Sigma^{-1}((I_{r} + 2 \Sigma t)^{-\frac{1}{2}} - I_{r})\mathcal{U}^{T}A\mathcal{E}_{0}\\
&= \mathcal{E}_{0}(I_{J} - \mathcal{V}\mathcal{V}^{T} + \mathcal{V}(I_{r} + 2 \Sigma t)^{-\frac{1}{2}} \mathcal{V}^{T}),
\end{align*}
and so
\begin{align*}
\lim _{t\to \infty}  \mathcal{E}_{t} &= \mathcal{E}_{0} - (AC_{0})^{T}\mathcal{U}\Sigma^{-1}\mathcal{U}^{T}A\mathcal{E}_{0}\\
&= \mathcal{E}_{0}(I_{J} - \mathcal{V}\mathcal{V}^{T}),
\end{align*}
\end{corollary}

The empirical covariance can then be quantified as follows.

\begin{corollary}
Under Assumptions {\rm \ref{eq:A1}}--{\rm \ref{eq:A3}} and with $C_{0} = C_{L}C_{R}$ being any desired decomposition with $C_{L} \in \mathbb{R}^{n \times p}$, $C_{R} \in \mathbb{R}^{p \times n}$ for some $p \in \mathbb{N}^{+}$, the empirical covariance is given by
\begin{align*}
C_{t} &= C_{0} + C_{0}A^{T}\mathcal{U}\Sigma^{-1}((I_{r} + 2 \Sigma t)^{-1} - I_{r})\mathcal{U}^{T}AC_{0}\\
&= \frac{1}{J}\mathcal{E}_{0}(I_{J} - \mathcal{V}\mathcal{V}^{T} + \mathcal{V}(I_{r} + 2 \Sigma t)^{-1}\mathcal{V}^{T})\mathcal{E}_{0}^{T}\\
&= C_{L}(I_{n} + 2 C_{R}A^{T}AC_{L} t)^{-1}C_{R} \quad \text{for any decomposition $C_{0} = C_{L}C_{R}$},
\end{align*}
and so
\begin{align*}
\lim _{t\to \infty}  C_{t} &= C_{0} - C_{0}A^{T}\mathcal{U}\Sigma^{-1}\mathcal{U}^{T}AC_{0}\\
&= \frac{1}{J}\mathcal{E}_{0}(I_{J} - \mathcal{V}\mathcal{V}^{T})\mathcal{E}_{0}^{T}.
\end{align*}
\begin{proof}
The first two of the above expressions for $C_{t}$ are direct corollaries of \ref{thm:dynamics_EKI}. The last expression for $C_{t}$ is a direct corollary of Theorem \ref{thm:general_dynamics_thm} after noting that
\begin{equation*}
\frac{d}{dt}C_{t} = -2C_{t}A^{T}AC_{t}.
\end{equation*}
\end{proof}
\end{corollary}
Finally, we summarise the results concerning ensemble collapse.

\begin{corollary}[Ensemble Collapse in State Space]\label{ensemble_collapse}
Under Assumptions {\rm \ref{eq:A1}}--{\rm \ref{eq:A3}} and with $C_{0} = C_{L}C_{R}$ being any desired decomposition with $C_{L} \in \mathbb{R}^{n \times p}$, $C_{R} \in \mathbb{R}^{p \times n}$ for some $p \in \mathbb{N}^{+}$, we have
\begin{align*}
A\mathcal{E}_{t} &=  \mathcal{U}(I_{r} + 2 \Sigma t)^{-\frac{1}{2}}\mathcal{U}^{T}A\mathcal{E}_{0}\\
&= \sqrt{J}\mathcal{U} \Sigma^{\frac{1}{2}}(I_{r} + 2 \Sigma t)^{-\frac{1}{2}} \mathcal{V}^{T},
\end{align*}
and thus
\begin{equation*}
\lim_{t\to \infty} A\mathcal{E}_{t} = 0.
\end{equation*}
Similarly,
\begin{align*}
AC_{t} &= \mathcal{U}(I_{r} + 2 \Sigma t)^{-1}\mathcal{U}^{T}AC_{0}\\
&= \frac{1}{\sqrt{J}}\mathcal{U}\Sigma^{\frac{1}{2}}(I_{r} + 2 \Sigma t)^{-1}\mathcal{V}^{T}\mathcal{E}_{0}^{T},
\end{align*}
and thus
\begin{equation*}
\lim_{t\to \infty} AC_{t} = 0.
\end{equation*}
\end{corollary}

\subsection{Optimal Initial Emsemble}\label{subsec:adaptiveEKI}

The subspace property (Lemma \ref{lem:invariant_subspace}) of EKI ensures that the particles remain within the span of the initial ensemble throughout the inversion process. This property can be used to regularise the inverse problem, especially when the problem is ill-posed and cannot feasibly be inverted across the entire space. The initial ensemble is therefore a critical design parameter that should be carefully chosen to ensure accurate and computationally feasible results. Typically, prior knowledge is used to inform the choice of initial ensemble. 

We illustrate this concept using the linear Gaussian case, where the forward operator is linear, and both the prior and noise are Gaussian. The optimisation problem of finding the MAP estimate - i.e. choosing u to maximise the posterior probability - is equivalent to the following minimisation problem:

\begin{equation}\label{eq:tikhonov objective}
\min_{u\in \mathbb{R}^n} \frac{1}{2}|y - A(u)|^{2} + \frac{1}{2}|u|_{R}^{2} = \min_{u\in \mathbb{R}^n} \frac{1}{2}\left|\begin{pmatrix} Au \\ R^{-\frac{1}{2}}u\end{pmatrix} - \begin{pmatrix}y \\0 \end{pmatrix}\right|^{2}.
\end{equation}

Typically, the dominant eigenvectors of a given prior are used to construct the smaller parameter subspace to which inversion process is confined. Specifically, one selects the first J terms in the Karhunen-Lo\`{e}ve expansion of the Gaussian prior. If the covariance matrix $R$ of the prior has a diagonal decomposition $R := \mathcal{V} \Lambda \mathcal{V}^{T}$, where $\Lambda := \text{diag}(\{\lambda_{(i)}\}_{i}) \in \mathbb{R}^{n \times n}$ contains the eigenvalues of $R$, and $\mathcal{V} := [v_{(1)}, \ldots, v_{(n)}] \in \mathbb{R}^{n \times n}$ consists of the corresponding eigenvectors, then the truncated Karhunen-Lo\`{e}ve expansion is given by

\begin{equation*}
\sum_{i \in \mathcal{J}} \xi_{i}\sqrt{\lambda_{(i)}}v_{(i)} \sim \mathcal{N}(0, R),
\end{equation*}

where $\xi_{i} \sim \mathcal{N}(0, 1)$ i.i.d. and $\mathcal{J} \subseteq \{1, \ldots, n\}$ indexes the J largest eigenvalues of $R$. The initial ensemble is then taken to be

\begin{equation*}
U_{0} = [u^{(1)}(0), \ldots, u^{(J)}(0)] := [v_{(\mathcal{J}_{1})}, \ldots, v_{(\mathcal{J}_{J})}] \text{diag}(\{\sqrt{\lambda_{(i)}}\}_{i \in \mathcal{J}}) =:\mathcal{V}_{\mathcal{J}}\Lambda_{\mathcal{J}}^{\frac{1}{2}}.
\end{equation*}

In this approach, one then seeks to minimise the objective function in Equation \eqref{eq:tikhonov objective} by running EKI with forward operator $\begin{pmatrix} (Au)^{T} & (R^{-\frac{1}{2}}u)^{T}\end{pmatrix}^{T}$, data $\begin{pmatrix} y^{T} & 0^{T}\end{pmatrix}^{T}$ and initial ensemble $U_{0} = \mathcal{V}_{\mathcal{J}}\Lambda_{\mathcal{J}}^{\frac{1}{2}}$.

We seek to generalise this approach by choosing among any desired set of J eigenvectors of $R$ and by allowing the initial ensemble to consist of any desired linearly independent linear combination of these eigenvectors. Specifically, we define
\begin{equation*}
U_{\mathcal{J}, B}(0) = [u_{\mathcal{J}, B}^{(1)}(0), \ldots, u_{\mathcal{J}, B}^{(J)}(0)] := [v_{\mathcal{J}_{1}}, \ldots, v_{\mathcal{J}_{J}}]B =: \mathcal{V}_{\mathcal{J}}B,
\end{equation*}

Where $\mathcal{J} \subseteq \{1, \ldots, n\}$ is a subset of size J and $B \in \R^{J \times J}$ is an invertible matrix. Our objective is then to select $\mathcal{J}$ and B to optimise the long-term performance of EKI.

We know from Lemma \ref{ensemble_collapse} that $\forall i: Au_{\mathcal{J}, B}^{(i)}(t)$ converge to the same value. Our goal, therefore, is to choose $\mathcal{J} \subseteq \{1, \ldots, n\}$ and B to minimise

\begin{equation*}
\lim_{t \rightarrow \infty} \frac12|A\bar{u}_{\mathcal{J}, B}(t) - y|^{2} + \frac12|\bar{u}_{\mathcal{J}, B}(t)|_{R}^{2}.
\end{equation*}

In fact, to develop an EKI algorithm with adaptive resampling (see Section \ref{sec:Nonlinear Forward Operator and Resampling}), we later need to choose $\mathcal{J} \subseteq \{1, \ldots, n\}$ and B to minimise the more general objective function

\begin{equation*}\label{eq:general objective}
\lim_{t \rightarrow \infty} \frac12|A\bar{u}_{\mathcal{J}, B}(t) - y|^{2} + \frac12|\bar{u}_{\mathcal{J}, B}(t) - \mu|_{R}^{2}.
\end{equation*}

For any desired $\mu \in \R^{n}$. To this end we have the following theorem:

\begin{theorem}\label{thm:main thm}
Suppose that Assumptions  {\rm \ref{eq:A1}}--{\rm \ref{eq:A3}} hold. Further, let  $R \in \mathbb{R}^{n \times n}$ be symmetric and positive definite with an eigen-decomposition $R = \mathcal{V} \Lambda \mathcal{V}^{T} \in \mathbb{R}^{n \times n}$, where $\Lambda := \mbox{diag}(\{\lambda_{(i)}\}_{i}) \in \mathbb{R}^{n \times n}$ and $\mathcal{V} := [v_{1}, \ldots , v_{n}] \in \mathbb{R}^{n \times n}$. We denote by $\mathcal{J} \subseteq \{1, \ldots , n \}$ an index set of size J and $\mathcal{J}^{c} := \{1, \ldots , n \} \backslash \mathcal{J}$, and set $\mathcal{V}_{\mathcal{J}} := [v_{\mathcal{J}_{1}}, \ldots, v_{\mathcal{J}_{J}}] \in \mathbb{R}^{n \times J}$ and $\Lambda_{\mathcal{J}} := \mbox{diag}(\{ \lambda_{\mathcal{J}_{1}}, \ldots, \lambda_{\mathcal{J}_{J}}\})$. 

For an invertible matrix $B \in \mathbb{R}^{J \times J}$, we denote by $\{u^{(i)}_{\mathcal{J}, B}(t)\}_{i=1}^{J} \in \mathbb{R}^{n}$ the particles moving under deterministic EKI with the linear forward operator $\begin{pmatrix} A^{T} & R^{-\frac{1}{2}} \end{pmatrix}^{T} \in \mathbb{R}^{(m + n) \times n}$ and data $\begin{pmatrix} y^{T} & (R^{-\frac{1}{2}}\mu)^{T} \end{pmatrix}^{T} \in \mathbb{R}^{m + n}$ and with starting position $[u^{(1)}_{\mathcal{J}, B}(0),\ldots, u^{(J)}_{\mathcal{J}, B}(0)] := \mathcal{V}_{\mathcal{J}}B \in \R^{n \times J}$ , i.e. following the dynamics given by ODEs \eqref{eq:IVPEKI}. 

Let $1_{J} \in \mathbb{R}^{J}$ be a vector of ones, $A_{\mathcal{J}} := A\mathcal{V}_{\mathcal{J}} \in \mathbb{R}^{m \times J}$, $Q_{\mathcal{J}} := (I_{m} + A_{\mathcal{J}}\Lambda_{\mathcal{J}}A_{\mathcal{J}}^{T})^{-1} \in \R^{m \times m}$, $M_{\mathcal{J}} := (\Lambda_{\mathcal{J}}^{-1} + A_{\mathcal{J}}^{T}A_{\mathcal{J}} )^{-1} \in \R^{J \times J}$, $R_{\mathcal{J}}^{\dagger} := \mathcal{V}_{\mathcal{J}}\Lambda_{\mathcal{J}}^{-1}\mathcal{V}_{\mathcal{J}}$, $R_{\mathcal{J}^{c}}^{\dagger} := \mathcal{V}_{\mathcal{J}^{c}}\Lambda_{\mathcal{J}^{c}}^{-1}\mathcal{V}_{\mathcal{J}^{c}} \in \R^{m \times m}$ and $z_{\mathcal{J}} := y -  A\mathcal{V}_{\mathcal{J}}\mathcal{V}_{\mathcal{J}}^{T}\mu \in \R^{m}$.

Then $u^{(i)}_{\mathcal{J}, B}(t)$ converge for $i = 1, \ldots, J$ and 
\begin{align*}
\lim_{t \rightarrow \infty} &|A\bar{u}_{\mathcal{J}, B}(t) - y|^{2} + |\bar{u}_{\mathcal{J}, B}(t) - \mu|_{R}^{2} \\
&\quad = z_{\mathcal{J}}^{T}Q_{\mathcal{J}}z_{\mathcal{J}} + \frac{(1 - 1_{J}^{T}B^{-1}(\mathcal{V}_{\mathcal{J}}^{T}\mu + \Lambda_{\mathcal{J}}A_{\mathcal{J}}^{T}Q_{\mathcal{J}}z_{\mathcal{J}}))^{2}}{ 1_{J}^{T}B^{-1}\Lambda_{\mathcal{J}}B^{-T}1_{J} - (A_{\mathcal{J}}\Lambda_{\mathcal{J}}B^{-T}1_{J})^{T}Q_{\mathcal{J}}A_{\mathcal{J}}\Lambda_{\mathcal{J}}B^{-T}1_{J}} + \mu^{T}R^{\dagger}_{\mathcal{J}^{c}}\mu\\
&\quad= z_{\mathcal{J}}^{T}z_{\mathcal{J}} - z_{\mathcal{J}}^{T}A_{\mathcal{J}}M_{\mathcal{J}}A_{\mathcal{J}}^{T}z_{\mathcal{J}} + \frac{(1 - 1_{J}^{T}B^{-1}M_{\mathcal{J}}\mathcal{V}_{\mathcal{J}}^{T}(A^{T}y + R_{\mathcal{J}}^{\dagger}\mu))^{2}}{1_{J}^{T}B^{-1}M_{\mathcal{J}}B^{-T}1_{J}} + \mu^{T}R^{\dagger}_{\mathcal{J}^{c}}\mu 
.
\end{align*}
\end{theorem}
The proof of the theorem can be found in Appendix \ref{proof:main thm}.

\begin{corollary}\label{cor:optimality_condition}
Under the assumptions of Theorem \ref{thm:main thm}, the following inequality holds:
\begin{equation*}
\lim_{t \rightarrow \infty}\norma{A\bar{u}_{\mathcal{J}, B}(t) - y}^{2} + \norma{\bar{u}_{\mathcal{J}, B}(t) - \mu}_{R}^{2} \ge z_{\mathcal{J}}^{T}z_{\mathcal{J}} - z_{\mathcal{J}}^{T}A_{\mathcal{J}}M_{\mathcal{J}}A_{\mathcal{J}}^{T}z_{\mathcal{J}} + \mu^{T}R^{\dagger}_{\mathcal{J}^{c}}\mu,
\end{equation*}
where the lower bound is attained if and only if 
\begin{equation}\label{eq:optimal B condition}
1_{J}^{T}B^{-1}M_{\mathcal{J}}\mathcal{V}_{\mathcal{J}}^{T}(A^{T}y + R_{\mathcal{J}}^{\dagger}\mu) = 1.
\end{equation}
\end{corollary}

This optimality condition for the matrix B allows us to characterise the EKI estimate as the best approximation within the given subspace.

\begin{corollary}\label{limit_point_optimality}
Under the assumptions of Theorem \ref{thm:main thm}, and assuming that $B \in \mathbb{R}^{J \times J}$ satisfies the optimality condition given in Corollary \ref{cor:optimality_condition}, the following holds:
\begin{equation*}
\lim_{t \rightarrow \infty} u^{(i)}(t) = \argmin_{u \in \text{Im}(V_{\mathcal{J}})}  \frac{1}{2}\norma{Au - y}^{2} + \frac{1}{2}\norma{u - \mu}_{R}^{2} .
\end{equation*}
\end{corollary}

\subsubsection{Optimal Linear Combination}

For any invertible $D \in \R^{J \times J}$ such that $1_{J}^{T}D^{-1}M_{\mathcal{J}}\mathcal{V}_{\mathcal{J}}^{T}(A^{T}y + R_{\mathcal{J}}^{\dagger}\mu) \ne 0$,\\ $B:=(1_{J}^{T}D^{-1}M_{\mathcal{J}}\mathcal{V}_{\mathcal{J}}^{T}(A^{T}y + R_{\mathcal{J}}^{\dagger}\mu))^{-1}D$ satisfies Equation \eqref{eq:optimal B condition}. As such, there is a lot of flexibility in constructing an optimal B.

Given this flexibility in choosing an optimal B  - i.e.  a B that  minimises the final value of the objective function in Equation \eqref{eq:general objective} - we can restrict our attention to those optimal B that further minimise the initial value of the objective function.
\begin{corollary}\label{cor:double optimal B condition}
Let
\begin{equation*}
    \mathcal{B} := \{B \in \mathbb{R}^{J \times J} : 1 = 1_{J}^{T}B^{-1}   M_{\mathcal{J}}\mathcal{V}_{\mathcal{J}}^{T}(A^{T}y + R_{\mathcal{J}}^{\dagger}\mu)\},
\end{equation*}
i.e. $\mathcal{B}$ is the set of matrices that satisfy Equation \eqref{eq:optimal B condition}. Let $B^{*} \in \mathcal{B}$ satisfy
\begin{equation}\label{eq:double optimal B condition}
    B^{*}1_{J}/\sqrt{J} = \sqrt{J} M_{\mathcal{J}}\mathcal{V}_{\mathcal{J}}^{T}(A^{T}y + R_{\mathcal{J}}^{\dagger}\mu).
\end{equation}
Then
\begin{equation*}
    \norma{A\bar{u}_{\mathcal{J}, B^{*}}(0) - y}^{2} + \norma{\bar{u}_{\mathcal{J}, B^{*}}(0) - \mu}_{R}^{2} = \min_{B \in \mathcal{B}}\norma{A\bar{u}_{\mathcal{J}, B}(0) - y}^{2} + \norma{\bar{u}_{\mathcal{J}, B}(0) - \mu}_{R}^{2}.
\end{equation*}
\end{corollary}
The proof of this corollary can be found in Appendix \ref{proof: double optimal B condition}.

\subsubsection{Optimal Index Set}

Having established methods for generating an optimal B given a choice of eigenvectors $\mathcal{J}$, we now seek to choose a $\mathcal{J}$  that minimises the remaining term in Theorem \ref{thm:main thm} , i.e. the lower bound in Corollary \ref{cor:optimality_condition}:
\begin{equation}\label{reduced_eqn_to_optimise}
z_{\mathcal{J}}^{T}z_{\mathcal{J}} - z_{\mathcal{J}}^{T}A_{\mathcal{J}}\mathcal{M}_{\mathcal{J}}A_{\mathcal{J}}^{T}z_{\mathcal{J}} + \mu^{T}R^{\dagger}_{\mathcal{J}^{c}}\mu .
\end{equation}

Given that there are $\binom{n}{J}$ possible choices for $\mathcal{J}$, solving the objective function \eqref{reduced_eqn_to_optimise} by brute force quickly becomes computationally intractable for large n. We instead seek to iteratively build up the index set J , i.e. for k < N : $\mathcal{J}_{k} \subseteq \{1, \ldots, n\}$ with $|\mathcal{J}_{k}| = k$ and $\mathcal{J}_{k+1} := \mathcal{J} \cup \{j^{*}\}$ for some $j^{*} \in \{1, \ldots, n\} \backslash \mathcal{J}_{k}$. Recalling the definitions of $z_{\mathcal{J}}$ and $\norma{\mu}_{\mathcal{J}^{c}}^{2}$, it is clear that these are easy to iteratively calculate  ($z_{\mathcal{J}_{k + 1}} = z_{\mathcal{J}_{k}} - Av_{j^{*}}v_{j^{*}}^{T}\mu$ and $\mu^{T}R^{\dagger}_{\mathcal{J}_{k + 1}^{c}}\mu  = \mu^{T}R^{\dagger}_{\mathcal{J}_{k}^{c}}\mu - \frac{1}{\lambda_{j^{*}}}\langle \mu,v_{j^{*}}\rangle ^{2}$).

It remains only to iteratively calculate $z_{\mathcal{J}}^{T}A_{\mathcal{J}}\mathcal{M}_{\mathcal{J}}A_{\mathcal{J}}^{T}z_{\mathcal{J}}$. If we define $P_{\mathcal{J}} := A\mathcal{V}_{\mathcal{J}}\Lambda_{\mathcal{J}}^{\frac{1}{2}}$ and $\tilde{M}_{\mathcal{J}} := (I + P_{\mathcal{J}}^{T}P_{\mathcal{J}})^{-1}$, then
\begin{align*}
A_{\mathcal{J}}M_{\mathcal{J}}A_{\mathcal{J}}^{T} &= A\mathcal{V}_{\mathcal{J}}(\Lambda_{\mathcal{J}}^{-1} + \mathcal{V}_{\mathcal{J}}^{T}A^{T}A\mathcal{V}_{\mathcal{J}})^{-1}\mathcal{V}_{\mathcal{J}}^{T}A^{T}\\
&=:P_{\mathcal{J}}\tilde{M}_{\mathcal{J}}P_{\mathcal{J}}^{T}.
\end{align*}
In this form, we can make use of the following lemma:
\begin{lemma}\label{lemma:greedy}
Let $x \in \mathbb{R}^{m}$, $P_{k} := [p_{1}, \ldots, p_{k}] \in \mathbb{R}^{m \times k}$, and $\tilde{M}_{k} := (I_{k} + P_{k}^{T}P_{k})^{-1} \in \mathbb{R}^{k \times k}$. Define $\Delta_{k + 1} := 1 + p_{k+1}^{T}p_{k+1} - p_{k+1}^{T}P_{k}\tilde{M}_{k}P_{k}^{T}p_{k+1}$. Then
\begin{align*}
\tilde{M}_{k + 1} &= \frac{1}{\Delta_{k+1}}\begin{pmatrix} \Delta_{k+1} \tilde{M}_{k} + (\tilde{M}_{k}P_{k}^{T}p_{k +1})(\tilde{M}_{k}P_{k}^{T}p_{k +1})^{T} & -\tilde{M}_{k}P_{k}^{T}p_{k+1}\\
-p_{k+1}P_{k}\tilde{M}_{k} & 1
\end{pmatrix}.\\
x^{T}P_{k+1}\tilde{M}_{k+1}P_{k+1}^{T}x  &=  x^{T}P_{k}\tilde{M}_{k}P_{k}^{T}x + \frac{1}{\Delta_{k+1}}(x^{T}P_{k}\tilde{M}_{k}P_{k}^{T}p_{k+1} - x^{T}p_{k+1})^{2}.
\end{align*}
\end{lemma}

The proof of this lemma can be found in Appendix \ref{proof:greedy}.\\

Using this lemma, we can iteratively calculate the objective function as follows. Let $\mathcal{J}_{k} \subseteq \{1, \ldots , n\}$ with $|\mathcal{J}_{k}| = k$ and $k < J$. Define $p_{j} := \lambda_{j}^{\frac{1}{2}}Av_{j}$, $P_{k} := [p_{\mathcal{J}_{1}}, \ldots, p_{\mathcal{J}_{k}}] \in \mathbb{R}^{m \times k}$, $\tilde{M}_{k} := (I_{k} + P_{k}^{T}P_{k})^{-1} \in \mathbb{R}^{k \times k}$, $\mathcal{V}_{k} = [v_{\mathcal{J}_{1}}, \ldots, v_{\mathcal{J}_{k}}] \in \R^{n \times k}$, $\Lambda_{k} = \mbox{diag}(\lambda_{\mathcal{J}_{1}}, 
\ldots, \lambda_{\mathcal{J}_{k}}) \in \R^{k \times k}$. Then

\begin{align*}
\norma{z_{\mathcal{J}_{k+1}}}^{2} + \mu^{T}R^{\dagger}_{\mathcal{J}^{c}_{k+1}}\mu &- z_{\mathcal{J}_{k+1}}^{T}A_{\mathcal{J}_{k+1}}M_{\mathcal{J}_{k+1}}A_{\mathcal{J}_{k+1}}^{T}z_{\mathcal{J}_{k+1}}\\
&= \norma{z_{\mathcal{J}_{k+1}}}^{2} + \mu^{T}R^{\dagger}_{\mathcal{J}^{c}_{k+1}}\mu - z_{\mathcal{J}_{k+1}}^{T}P_{\mathcal{J}_{k+1}}\tilde{M}_{\mathcal{J}_{k+1}}P_{\mathcal{J}_{k+1}}^{T}z_{\mathcal{J}_{k+1}}\\
&= \norma{z_{\mathcal{J}_{k+1}}}^{2} + \mu^{T}R^{\dagger}_{\mathcal{J}^{c}_{k+1}}\mu -z_{\mathcal{J}_{k+1}}^{T}P_{\mathcal{J}_{k}}\tilde{M}_{\mathcal{J}_{k}}P_{\mathcal{J}_{k}}^{T}z_{\mathcal{J}_{k+1}} \\
&\quad - \frac{(z_{\mathcal{J}_{k+1}}^{T}P_{\mathcal{J}_{k}}\tilde{M}_{\mathcal{J}_{k}}P_{\mathcal{J}_{k}}^{T}p_{\mathcal{J}_{k+1}} - z_{\mathcal{J}_{k+1}}^{T}p_{\mathcal{J}_{k+1}})^{2}}{1 + p_{k+1}^{T}p_{k+1} - p_{k+1}^{T}P_{\mathcal{J}_{k}}\tilde{M}_{\mathcal{J}_{k}}P_{\mathcal{J}_{k}}^{T}p_{k+1}}
\end{align*}

This leads naturally to the Algorithm \ref{alg:selectJ}, a greedy algorithm for iteratively selecting indices for the index set $\mathcal{J}$ as defined in Theorem \ref{thm:main thm}.

\begin{varalgorithm}{Greedy}
\caption{Greedy algorithm for selecting eigenvectors}\label{alg:selectJ}
\begin{algorithmic}
\State \textbf{Initialise: }
\State $k \leftarrow 1$
\State $p_{i} \leftarrow \lambda_{i}^{\frac{1}{2}}Av_{i} \in \mathbb{R}^{m}, \quad i = 1, \ldots, n$
\State $\mathcal{J} \leftarrow \{\}$
\State $\mathcal{J}^{c} \leftarrow \{1, \ldots, n\} $
\State $P \leftarrow []$
\State $\tilde{M} \leftarrow []$
\State $z \leftarrow y$
\While{k < J}
\State $k \leftarrow k + 1$
\State $\mbox{min. val.} \leftarrow +\infty$
\For{$j \in \mathcal{J}^{c}$}
\State  $\tilde{z} = z - (\mu^{T}v_{j})Av_{j}$
\State $\mbox{val. } \leftarrow -2(\mu^{T}v_{j})(z^{T} Av_{j}) + (\mu^{T}v_{j})^{2}(\norma{Av_{j}}^{2} - \frac{1}{\lambda{j}}) - \tilde{z}^{T}P\tilde{M}P^{T}\tilde{z} -\frac{(\tilde{z}^{T}P\tilde{M}P^{T}p_{i} - \tilde{z}^{T}p_{i})^{2}}{1 + p_{i}^{T}p_{i} - p_{i}^{T}P\tilde{M}P^{T}p_{i}}$
\If{\mbox{val.} < \mbox{min. val.}}
\State $\mbox{min. val.} \leftarrow \mbox{val.}$
\State $j^{*} \leftarrow j$
\EndIf
\EndFor
\State $\mathcal{J} \leftarrow \mathcal{J} \cup \{j^{*}\}$
\State $\mathcal{J}^{c} \leftarrow \mathcal{J}^{c} \backslash \{j^{*}\}$
\State $\Delta \leftarrow 1 + p_{j^{*}}^{T}p_{j^{*}} - p_{j^{*}}^{T}P\tilde{M}P^{T}p_{j^{*}}$
\State $\tilde{M} \leftarrow \frac{1}{\Delta}\begin{pmatrix} \Delta \tilde{M} + (\tilde{M}P^{T}p_{j^{*}})(\tilde{M}P^{T}p_{j^{*}})^{T} & -\tilde{M}P^{T}p_{j^{*}}\\
-p_{j^{*}}^{T}P\tilde{M} & 1
\end{pmatrix} \in \mathbb{R}^{k \times k}$
\State $P \leftarrow \begin{bmatrix} P & p_{j^{*}}\end{bmatrix} \in \mathbb{R}^{m \times k}$
\State $z \leftarrow z - \langle\mu, v_{j^{*}}\rangle Av_{j^{*}}$
\EndWhile
\State $M_{\mathcal{J}} \leftarrow \Lambda_{\mathcal{J}}^{\frac{1}{2}}\tilde{M}\Lambda_{\mathcal{J}}^{\frac{1}{2}}$
\end{algorithmic}
\end{varalgorithm}

This algorithm does not, in general, select an optimal subset $\mathcal{J}$. We do, however, show a sufficient condition on the forward operator $A$ that ensures that Algorithm \ref{alg:selectJ} selects an optimal $\mathcal{J}$.

\begin{proposition}\label{proposition:optimal_case}
 Under the definitions and assumptions of Theorem \ref{thm:main thm}, if $R = \mathcal{V}\Lambda\mathcal{V}^{T}$ and A has singular value decomposition given by $A = \mathcal{U}\Sigma\mathcal{V}^{T}$, then Algorithm \ref{alg:selectJ} finds an optimal subset $\mathcal{J}$ , i.e. a $\mathcal{J}$ that minimises $z_{\mathcal{J}}^{T}z_{\mathcal{J}} - z_{\mathcal{J}}^{T}A_{\mathcal{J}}\mathcal{M}_{\mathcal{J}}A_{\mathcal{J}}^{T}z_{\mathcal{J}} + \mu^{T}R_{\mathcal{J}^{c}}^{\dagger}\mu$.
\end{proposition}
The proof of the proposition can be found in Appendix \ref{proof:optimal_case}.

\subsection{Nonlinear Forward Operator and Resampling}\label{sec:Nonlinear Forward Operator and Resampling}
We now generalise our approach to allow resampling of the ensemble and to handle nonlinear forward operators. Suppose we wish to use EKI to generate a set of particle trajectories $\{u_{t}^{(i)}\}_{i = 1}^{J}$ to minimise the usual objective function
\begin{equation*}
\frac{1}{2}\norma{G(u) - y}^{2} + \frac{1}{2}\norma{u}_{R}^{2}.
\end{equation*}
Suppose at $t_{0}$ we wish to resample the J particles. We write $\{u^{(i)}_{t_{0}^{-}} \}_{i = 1}^{J}$ for the particle positions at time $t_0$ before resampling and $\{u^{(i)}_{t_{0}^{+}} \}_{i = 1}^{J}$ for the particle positions at time $t_0$ after resampling. One approach to resampling would be to linearise about $\bar{u}_{t_{0}^{-}}$ in order to apply our ensemble selection results for linear forward operators , i.e.
\begin{align}\label{eq:linearised objective}
\frac{1}{2}\norma{G(\bar{u}_{t_{0}} + \tilde{u}) - y}^{2} + \frac{1}{2}\norma{\bar{u}_{t_{0}} + \tilde{u}}_{R}^{2} &\approx \frac{1}{2}\norma{\nabla G(\bar{u}_{t^{-}_{0}})\tilde{u} - (y - G(\bar{u}_{t^{-}_{0}}))}^{2} + \frac{1}{2}\norma{\bar{u}_{t^{-}_{0}} + \tilde{u}}_{R}^{2} \nonumber\\
&= \frac{1}{2}\norma{A\tilde{u} - \tilde{y}}^{2} + \frac{1}{2}\norma{ \tilde{u} - \mu}_{R}^{2},
\end{align}
where $A := \nabla G(\bar{u}_{t^{-}_{0}})$, $\tilde{y} := y -  G(\bar{u}_{t^{-}_{0}})$ and $\mu := -\bar{u}_{t^{-}_{0}}$. With these defined parameters, we can use Algorithm \ref{alg:selectJ} and Corollary \ref{cor:optimality_condition} to generate $\mathcal{J}$, B and set $u^{(i)}_{t_{0}^{+}} := \bar{u}_{t^{-}_{0}} + \tilde{u}^{(i)}_{t_{0}} := \bar{u}_{t^{-}_{0}} + \mathcal{V}_{\mathcal{J}}Be_{i}$ before continuing to run EKI.

This is a heuristic approach, which linearises a general forward operator in order to apply the linear results. Since we linearise around the point $\bar{u}_{t_{0}^{-}}$, we may want the resampled points $\{ u^{(i)}_{t_{0}^{+}} \}_{i=1}^{J}$ to remain close, so that the linearisation is a good approximation. The total (squared) distance to $\bar{u}_{t_0^{-}}$ is given by
\begin{equation*}
\sum_{i = 1}^{J} \norma{u^{(i)}_{t^{+}_{0}} - \bar{u}^{(i)}_{t^{-}_{0}}}^{2} = \sum_{i = 1}^{J} \norma{\mathcal{V}_{\mathcal{J}}Be_{i}}^{2} = \mbox{trace}(B^{T}B).
\end{equation*}
As such, we seek a B that solves the following minimisation problem:
\begin{equation*}
\min\mbox{trace}(B^{T}B)\mbox{ : }1 = 1_{J}^{T}B^{-1}M_{\mathcal{J}}\mathcal{V}_{\mathcal{J}}^{T}(A^{T}\tilde{y} + R_{\mathcal{J}}^{\dagger}\mu).\\
\end{equation*}
To this end, Lemma \ref{lemma:minimal trace} provides an approximate solution to this minimisation problem. Namely, let $\tilde{z} := M_{\mathcal{J}}\mathcal{V}_{\mathcal{J}}^{T}(A^{T}\tilde{y} + R_{\mathcal{J}}^{\dagger}\mu)$ and let $U \in \mathbb{R}^{J \times J}$ be an orthogonal matrix such that $U1_{J} / \sqrt{J}  =   \tilde{z}/\norma{\tilde{z}}$. Then 
\begin{equation*}
B := \sqrt{J} \norma{\tilde{z}}U.
\end{equation*}
In this case we have that 
\begin{equation*}
\mbox{trace}(B^{T}B) = J^{2} \norma{\tilde{z}}^{2}.
\end{equation*}

Alternatively, we can choose B to instead approximately minimise the total (squared) distance to $\bar{u}_{t_0^{-}}$ in projection space, i.e.
\begin{equation*}
\sum_{i=1}^{J}\norma{\begin{pmatrix} G(u^{(i)}_{t_{0}^{+}}) \\ R^{-\frac{1}{2}}u^{(i)}_{t_{0}^{+}}\end{pmatrix} - \begin{pmatrix} G(u^{(i)}_{t_{0}^{-}}) \\ R^{-\frac{1}{2}}u^{(i)}_{t_{0}^{-}}\end{pmatrix}}^{2} \approx \sum_{i=1}^{J}\norma{\begin{pmatrix} \nabla G(\bar{u}^{(i)}_{t_{0}})V_{\mathcal{J}}B \\ R^{-\frac{1}{2}}V_{\mathcal{J}}B\end{pmatrix}}^{2} = \mbox{trace}(B^{T}M_{\mathcal{J}}^{-1}B).
\end{equation*}

This leads to the following minimisation problem
\begin{equation*}
\min\mbox{trace}(B^{T}\mathcal{M}_{\mathcal{J}}^{-1}B)\mbox{ : }1 = 1_{J}^{T}B^{-1}M_{\mathcal{J}}\mathcal{V}_{\mathcal{J}}^{T}(A^{T}\tilde{y} + R_{\mathcal{J}}^{\dagger}\mu).\\
\end{equation*}

Again, Lemma \ref{lemma:minimal trace} provides an approximate solution to this minimisation problem. Let $U \in \mathbb{R}^{J \times J}$ be an orthogonal matrix such that $U1_{J} / \sqrt{J}  =   M_{\mathcal{J}}^{-\frac{1}{2}}\tilde{z}/\norma{M_{\mathcal{J}}^{-\frac{1}{2}}\tilde{z}}$. Then

\begin{equation*}
B := \sqrt{J} \norma{M_{\mathcal{J}}^{-\frac{1}{2}}\tilde{z}}M_{\mathcal{J}}^{\frac{1}{2}}U.
\end{equation*}
In this case we have that 
\begin{equation*}
\mbox{trace}(B^{T}\mathcal{M}_{\mathcal{J}}^{-1}B) = J^{2} \norma{M_{\mathcal{J}}^{\frac{1}{2}}\tilde{z}}^{2}.
\end{equation*}

In either cases, it is clear that these Bs chosen to minimise the various traces are optimal both in the sense of Corollary \ref{cor:optimality_condition} and the stronger sense of Corollary \ref{cor:double optimal B condition}.

These results lead to Algorithm \ref{alg:adaptive ensemble selection} for performing EKI for  $t \in [0, t_{N + 1}]$ with resampling at times $t_{1} < \dots < t_{N}$.
\begin{varalgorithm}{AdaptSelect}
\caption{Adaptive ensemble selection}\label{alg:adaptive ensemble selection}
\begin{algorithmic}
\State \textbf{Initialise: }
\State $\mu \mbox{ e.g. } \mu \leftarrow 0$\\
\State $0 = t_0 < \ldots < t_{N} < t_{N + 1}$\\
\For{k = 1:N }
\State $A \leftarrow \nabla  G(-\mu)$\\
\State $\tilde{y} \leftarrow y -  G(-\mu)$\\
\State $\mathcal{J}, M_{\mathcal{J}} \leftarrow Algorithm \ref{alg:selectJ}(A, \tilde{y}, \mu)$\\
\State $\tilde{z} \leftarrow M_{\mathcal{J}}\mathcal{V}_{\mathcal{J}}^{T}(A^{T}\tilde{y} + R_{\mathcal{J}}^{\dagger}\mu)$\\
\State $B \leftarrow \sqrt{J}\norma{\tilde{z}}_{2}Householder(1_{J} / \sqrt{J}, \quad \tilde{z} / \norma{\tilde{z}})$\\
\State $[u^{(1)}, \ldots, u^{(J)}] \leftarrow -[\mu, \ldots , \mu] + \mathcal{V}_{\mathcal{J}}B$\\
\State $\{u^{(i)}\}_{i = 1}^{J} \leftarrow EKI(t_{k+1} - t_{k}, \{u^{(i)}\}_{i = 1}^{J}, G, R, y)$\\
\State $\mu \leftarrow -\frac{1}{J}\sum_{i=1}^{J}u^{(i)}$\\
\EndFor\\
\end{algorithmic}
\end{varalgorithm}

Algorithm \ref{alg:adaptive ensemble selection} is a largely derivative-free method, only requiring derivative calculations at initialisation and at resample times. For problems in which derivative calculations are expensive, Algorithm \ref{alg:adaptive ensemble selection} with minimal resampling could therefore be significantly more efficient computationally than standard optimisation methods that rely heavily on derivative calculations. In fact, if we were to use a derivative-free linearisation in Equation \eqref{eq:linearised objective}, Algorithm \ref{alg:adaptive ensemble selection} itself would become a completely derivative-free method.
 
\section{Numerical Experiments}\label{sec:NumExp}
In the following, we numerically test the proposed EKI methods for linear and nonlinear problems. 
\subsection{Set-up}\label{subsubsection:set-up}
We run our experiments according to the following generalised scheme. We first generate the true solution $u_{\text{true}}$ and define the observation data as $y :=  G(u_{\text{true}}) + \eta$, where $\eta$ represents observational noise. Using the same data, we execute variants of Algorithm \ref{alg:adaptive ensemble selection} to minimise the usual objective function

\begin{equation*}\label{eq:weighted objective}
\Phi(u) := \frac{1}{2}\norma{ G(u) - y}^{2} + \frac{\beta}{2}\norma{u}_{R}^{2}.
\end{equation*}

The variants of Algorithm \ref{alg:adaptive ensemble selection} differ in their (re-)sampling times $0 = t_0 < \ldots < t_N$, in their methods for selecting $\mathcal{J}$ and in their methods for constructing B. In our numerical experiments, we investigate the following variants:
\begin{itemize}
    \item \textbf{greedy} (greedy): at $t = t_{0}$, $\mathcal{J}$ is selected according to Algorithm \ref{alg:selectJ} and B is constructed according to Corollary \ref{cor:double optimal B condition}.
    \item \textbf{dominant} (dom): $\mathcal{J}$ consists of the indices of the dominant eigenvalues and B is constructed according to \ref{cor:double optimal B condition}. This variant is included to evaluate, by way of comparison with the standard method (defined below), the effect of constructing B instead according to Corollary \ref{cor:double optimal B condition}. 
    \item \textbf{random} (rand): at $t = t_{0}$, $\mathcal{J}$ is chosen randomly and B is constructed according to Corollary \ref{cor:double optimal B condition}. This variant is included as a baseline to evaluate, by way of comparison, the effect of different methods for selecting $\mathcal{J}$. 
    \item \textbf{standard} (stand): at $t = t_{0}$, $\mathcal{J}$ is chosen to be the indices of the dominant eigenvalues and $B := \Lambda_{\mathcal{J}}^{\frac{1}{2}}$, i.e. the standard initialisation method.
    \item \textbf{optimal} (opt): this is only included if both the forward operator is linear and $\begin{pmatrix} n\\J \end{pmatrix}$ is small. For every possible index set $\mathcal{J} \subseteq \{ 1, \dots , n\}$ of size J, B is constructed according to Corollary \ref{cor:double optimal B condition} and the resulting long-term value of the objective function is calculated. $\mathcal{J}$ and B are then taken to be the best performing pair of index set and linear combination. This variant is included to evaluate, by way of comparison, the performance of different methods for selecting $\mathcal{J}$.
    \item \textbf{greedy resampling} (greedy\_r): at times $t = t_{0}, \ldots, t_{N}$, $\mathcal{J}$ is selected according to Algorithm \ref{alg:selectJ} and B is constructed according to Corollary \ref{cor:double optimal B condition}.
    \item \textbf{dominant resampling} (dom\_r): at times $t = t_{0}, \ldots, t_{N}$, $\mathcal{J}$ consists of the indices of the dominant eigenvalues, while B is constructed according to Corollary \ref{cor:double optimal B condition}. This variant is included to evaluate, by way of comparison with the standard method, the performance of the selecting B according to Corollary \ref{cor:double optimal B condition}.
\end{itemize}

We also investigate the performance of the different variants of the algorithm across various values for the following parameters:

\begin{itemize} 
\item \textbf{J} : The number of particles.
\item $\boldsymbol{\beta}$ : The relative weight of the prior/regularisation term in the objective function given in Equation \eqref{eq:weighted objective}.
\end{itemize}

The value of the objective function in Equation \eqref{eq:weighted objective} at the final time $t_{N + 1}$ for each variant is denoted by an $r$ with a superscript indicating the variant, e.g. $r^{\text{\tiny greedy}} := \Phi(\bar{u}_{t_{N + 1}}^{\text{\tiny greedy}})$.

For comparison, we also define $r^{\text{\tiny min}} := \min_{u}\Phi(u)$. In the case that the forward operator is linear, $r^{\text{\tiny min}}$ is explicitly calculated; otherwise, it is be approximated using a numerical optimiser (more specifically, the fmincon MATLAB optimiser). We present our results as a ratio, e.g. $\frac{r^{\text{\tiny min}}}{r^{\text{\tiny greedy}}}$. 

For a given variant of the algorithm and a given parameter configuration, we evaluate its performance across a range of random data points, and potentially across a range of forward operators and/or covariance matrices. As such, for each parameter configuration, we run $n_{exp}$ experiments, in each of which we generate a random data point, and potentially a random forward operator and/or covariance matrix, and run each variant of the EKI method on this given data point. The results of these experiments are indexed by $1 \le i \le n_{\scalebox{0.75}{\mbox{exp}}}$, e.g. $\frac{r^{\text{\scalebox{1}{min}}}_{i}}{r^{\text{\scalebox{1}{greedy}}}_{i}}$ for $i = 1, \ldots, n_{ \scalebox{0.75}{\mbox{exp}}}$. We then average across these experiments, e.g. $\frac{r^{\text{\scalebox{1}{min}}}}{r^{\text{\scalebox{1}{greedy}}}} := \frac{1}{n_{\scalebox{0.5}{\mbox{exp}}}}\sum_{i=1}^{n_{\scalebox{0.5}{\mbox{exp}}}}\frac{r^{\text{\scalebox{1}{min}}}_{i}}{r^{\text{ \scalebox{1}{greedy}}}_{i}}$.

To assess the quality of a given index set $\mathcal{J}$ - e.g. $\mathcal{J}^{\text{\scalebox{1}{greedy}}}$ - for each of the $n_{\scalebox{0.75}{\mbox{exp}}}$ experiments (indexed by i), we may also generate $n_{\scalebox{0.5}{\mbox{rand}}}$  random subsets of $\{1, \ldots, n\}$ of size J, calculate the resulting values of the objective (indexed by j), and compute the percentage of $r_i^{\text{\tiny rand}_j}$ values greater than the objective function of a given variant, e.g. $r_i^{\text{\tiny greedy}}$. This is denoted as, for example, $\% \, r_i^{\text{\tiny rand}} \geq r_i^{\text{\tiny greedy}}$. Averaging over experiments, we define, for example, $\% \, r^{\text{\tiny rand}} \geq r^{\text{\tiny greedy}} := \frac{1}{n_{\scalebox{0.5}{\mbox{exp}}}}\sum_{i=1}^{n_{\scalebox{0.5}{\mbox{exp}}}}(\% \, r_i^{\text{\tiny rand}} \geq r_i^{\text{\tiny greedy}})$.

\subsection{Linear Forward Operator}
We focus first on the linear setting.
\subsubsection{Experiment Set-up}\label{subsubsec:linear set-up}

We conduct 250 numerical experiments ($n_{\scalebox{0.75}{\mbox{exp}}} = 250$). In each experiment, a minimisation problem is randomly generated with the following objective function:

\begin{equation}\label{linear_obj}
\Phi_{i}(u) = \frac{1}{2}\norma{A_i u - y_i}^2 + \frac{\beta}{2}\norma{u}_{R_i}^2 \quad \text{for }, i = 1, \dots, 250.
\end{equation}

The elements of the linear forward operators $A_i \in \mathbb{R}^{30 \times 50}$ are independently and identically distributed (i.i.d.) according to a uniform distribution on $[0, 1]$.

To generate $R_i$, we first sample $P_i \in O(50)$ uniformly from the space of real orthogonal $50 \times 50$ matrices $O(50)$, and then define $R_i := P_i \Sigma P_i^\top$, where $\Sigma$ is a diagonal matrix with diagonal entries decaying according to

\begin{equation*}
\sigma_{k} = (1 + k)^{-2}, \quad k = 1, \dots, 50.    
\end{equation*}

To generate $y_i \in \mathbb{R}^{30}$, we first randomly generate $u_i \in \mathbb{R}^{50}$ by sampling i.i.d. from $\mathcal{N}(0, R_i)$. We then set $y_i := A_i u_i + 10^{-4}\eta_i$, where $\eta_i \sim \mathcal{N}(0, I_{30})$ are i.i.d.

Additionally, for each experiment, we generate 200 random subsets of size J ($n_{\scalebox{0.75}{\mbox{rand}}} = 200$), yielding $r_i^{\text{\tiny rand}_j}$ for $j = 1, \dots, 200$. The average residual for these random subsets is defined as $r_i^{\text{\tiny rand}} := \frac{1}{200} \sum_{j=1}^{200} r_i^{\text{\tiny rand}_j}$.

To minimise the objective function given by \ref{linear_obj}, we apply the following variants of algorithm \ref{alg:adaptive ensemble selection}: \textbf{greedy}, \textbf{dom}, \textbf{stand}, \textbf{rand} and - when J is relatively small - \textbf{opt} (see Subsection \ref{subsubsection:set-up}). For each configuration of parameters J and $\beta$, 250 numerical experiments are run ($n_{\scalebox{0.75}{\mbox{exp}}}=250$). For each numerical experiment, a set of initial conditions are generated, and the resulting long-term value of the objective function is calculated for each of the above variants with these same initial conditions.

\subsubsection{Experiment Results}
\noindent
\def\arraystretch{1.5}
\begin{table}[H]
\caption{Results for increasing J, the number of particles, in the various EKI strategies with the regularised random linear forward operators described in Subsection \ref{subsubsec:linear set-up}.  $\beta = 10^{-4}$ for all experiments.}
\label{table:linear, J}
\begin{tabular}{|p{0.145\textwidth}|p{0.105\textwidth}|p{0.105\textwidth}|p{0.105\textwidth}|p{0.105\textwidth}|p{0.105\textwidth}|p{0.105\textwidth}|}
\hline
\multicolumn{6}{|c|}{\bf{Mean Residual Ratio}}\\
\hline
 &  \bf{\small J=2} & \bf{\small J=4} &  \bf{\small J=6} & \bf{\small J=8} & \bf{\small J=10}\\
\hline
$\boldsymbol{\frac{r^{\mbox{\tiny min}}}{r^{\mbox{\tiny greedy}}}}$ & \colorbox{lightgray}{0.0504} & \colorbox{lightgray}{0.115} & \colorbox{lightgray}{0.192} & \colorbox{lightgray}{0.269} & \colorbox{lightgray}{0.386 }\\
\hline
$\boldsymbol{\frac{r^{\mbox{\tiny min}}}{r^{\mbox{\tiny dom}}}}$ & 0.0315 & 0.0657 & 0.103 & 0.138 & 0.200 \\
\hline
$\boldsymbol{\frac{r^{\mbox{\tiny min}}}{r^{\mbox{\tiny stand}}}}$ & $8.62{\times}10^{-5}$  & $5.05{\times}10^{-4}$ & 0.00151 & 0.00319 & 0.00632\\
\hline
$\boldsymbol{\frac{r^{\mbox{\tiny min}}}{r^{\mbox{\tiny rand}}}}$ & 0.0137 & 0.0189 & 0.0232 & 0.0262 & 0.0295\\
\hline
\multicolumn{6}{|c|}{\bf{Mean Residual Percentile}}\\
\hline
$\scriptstyle \boldsymbol{\% \mbox{ } r^{\mbox{\tiny rand}} \ge r^{\mbox{\tiny greedy}}}$ & \colorbox{lightgray}{99.872} & \colorbox{lightgray}{99.996} & \colorbox{lightgray}{100} & \colorbox{lightgray}{100}  & \colorbox{lightgray}{100} \\
\hline
$\scriptstyle \boldsymbol{\% \mbox{ } r^{\mbox{\tiny rand}} \ge r^{\mbox{\tiny dom}}}$ & 82.614 & 95.462 & 98.154 & 99.344 & 99.816 \\
\hline
$\scriptstyle \boldsymbol{\% \mbox{ } r^{\mbox{\tiny rand}} \ge r^{\mbox{\tiny stand}}}$ & 0 & 0.002 & 0.416 & 0.49 & 2.75\\
\hline
\end{tabular}
\end{table}
\def\arraystretch{1.5}

\def\arraystretch{1.5}
\begin{table}[H]
\caption{Results for increasing $\beta$, the relative weight of the prior, in the various EKI strategies with the regularised random linear forward operators described in Subsection \ref{subsubsec:linear set-up}. J = 5 for all experiments.}
\label{table:linear, beta}
\begin{tabular}{|p{0.145\textwidth}|p{0.105\textwidth}|p{0.105\textwidth}|p{0.105\textwidth}|p{0.105\textwidth}|p{0.105\textwidth}|p{0.105\textwidth}|}
\hline
\multicolumn{7}{|c|}{\bf{Mean Residual Ratio}}\\
\hline
 &  $\boldsymbol{ \beta=10^{-3}}$ & $\boldsymbol{ \beta=10^{-2}}$ &  $\boldsymbol{ \beta=10^{-1}}$ & $\boldsymbol{ \beta=10^{0}}$ & $\boldsymbol{ \beta=10^{1}}$ & $\boldsymbol{ \beta=10^{2}}$\\
\hline
$\boldsymbol{\Large 
\frac{r^{\mbox{\tiny min}}}{r^{\mbox{\tiny opt}}}}$ & \colorbox{lightgray}{0.161} & \colorbox{lightgray}{0.572} &  \colorbox{lightgray}{0.889} & \colorbox{lightgray}{0.941} & \colorbox{lightgray}{0.977} & \colorbox{lightgray}{0.997}\\
\hline
$\boldsymbol{\Large 
\frac{r^{\mbox{\tiny min}}}{r^{\mbox{\tiny greedy}}}}$ & 0.15  & 0.562 & 0.888 & 0.941 & 0.977 & 0.997\\
\hline
$\boldsymbol{\frac{r^{\mbox{\tiny min}}}{r^{\mbox{\tiny dom}}}}$ & 0.0855 & 0.396 & 0.815 & 0.896 & 0.964 & 0.995\\
\hline
$\boldsymbol{\frac{r^{\mbox{\tiny min}}}{r^{\mbox{\tiny stand}}}}$ & 0.000909 & 0.0447 & 0.439 & 0.708 & 0.856 & 0.914\\
\hline
$\boldsymbol{\frac{r^{\mbox{\tiny min}}}{r^{\mbox{\tiny rand}}}}$ & 0.0209 & 0.0973 & 0.245 & 0.488 & 0.827 & 0.976\\
\hline
\multicolumn{7}{|c|}{\bf{Mean Residual Percentile}}\\
\hline
$\scriptstyle \boldsymbol{\% \mbox{ } r^{\mbox{\tiny rand}} \ge r^{\mbox{\tiny greedy}}}$ & \colorbox{lightgray}{99.992} & \colorbox{lightgray}{100} & \colorbox{lightgray}{100} & \colorbox{lightgray}{100} & \colorbox{lightgray}{100} & \colorbox{lightgray}{100}\\
\hline
$\scriptstyle \boldsymbol{\% \mbox{ } r^{\mbox{\tiny rand}} \ge r^{\mbox{\tiny dom}}}$ & 97.604 & 98.338 & 99.634 & 99.644 & 99.78 & 99.674\\
\hline
$\scriptstyle \boldsymbol{\% \mbox{ } r^{\mbox{\tiny rand}} \ge r^{\mbox{\tiny stand}}}$ & 0.006 & 18.452 & 71.68 & 74.53 & 55.672 & 21.644\\
\hline
\end{tabular}
\end{table}
\def\arraystretch{1.5}

\subsubsection{Discussion of Experiment Results}\label{subsubsec: discussion, linear}

Table \ref{table:linear, J} shows a monotonic increase in the ratios with the number of particles for all variants of the algorithm. That this should be the case is intuitively clear and can be mathematically confirmed using Corollary \ref{cor:optimality_condition}.

Table \ref{table:linear, beta} shows that $r^{\mbox{\tiny greedy}} \le r^{\mbox{\tiny dom}} \le r^{ \mbox{\tiny stand}}$ across all values of $\beta$.  For all variants of EKI, the average ratios of the long-term values of the objective function increases monotonically as the value of $\beta$ increases. $\frac{r}{r^{\mbox{\tiny greedy}}}$, $\frac{r}{r^{\mbox{\tiny dom}}}$ and $\frac{r}{\bar{r}}$ converge to 1, while $\frac{r}{r^{\mbox{\tiny stand}}}$ converges to a value strictly less than 1. This is expected since, for all EKI with an `optimal B', $M_{\mathcal{J}} \approx \frac{1}{\beta}\Lambda_{\mathcal{J}}$ and Corollary \ref{cor:optimality_condition} therefore shows that $r$, $r^{\mbox{\tiny greedy}}, r^{\mbox{\tiny dom}}, \bar{r} \rightarrow \|y\|^{2}$ as $\beta \rightarrow \infty$. For the standard method, $B = \Lambda_{\mathcal{J}}^{\frac{1}{2}}$, and so Corollary \ref{cor:optimality_condition} shows that $r^{\mbox{\tiny stand}} \rightarrow \|y\|^{2} + \frac{1}{J}$ as $\beta \rightarrow \infty$.

\subsection{Nonlinear Forward operator}
In the following, we apply the proposed EKI methods to two inverse problems with nonlinear forward operators.
\subsubsection{Nonlinear Algebraic Forward Operator}

We run $25$ experiments ($n_{\scalebox{0.75}{\mbox{exp}}} = 25$). For each experiment we randomly generate a minimisation problem with the following objective function
\begin{equation}\label{nonlinear_obj}
\Phi_{i}(u) = \frac{1}{2}\norma{G_{i}(u) - y_{i}}^{2} + \frac{\beta}{2}\norma{u}_{C_{i}}^{2}, \quad i = 1, \ldots , 25.
\end{equation}
To generate $G_{i}$, for each experiment, we randomly sample $W_{i} \in \mathbb{R}^{30 \times 50}$, the elements of which are i.i.d. generated according to a uniform distribution on $[0, 1]$. We then define
\begin{equation}
\label{eq:algebraic forward operator}
(G_{i}(u))_{j} = (1 + (\exp(W_{i}u)_{j}))^{-1}, \quad j = 1, \ldots, 30.
\end{equation}
\indent To generate the prior covariance $C_{i}$, we first generate $P_{i} \in O(50)$ by sampling uniformly from the space of real orthogonal $50 \times 50$ matrices $O(50)$, and then we define $C_{i} := P_{i}\Sigma P_{i}^{T}$, where $\Sigma$ is a diagonal matrix whose diagonal entries decay according to
\begin{equation*}
\sigma_{k} = (1 + k)^{-2}, \quad k = 1, \ldots , 50. 
\end{equation*}
\indent To generate $y_{i} \in \mathbb{R}^{30}$ we first randomly generate $u_{i} \in \mathbb{R}^{50}$ by sampling i.i.d. uniformly from $\mathcal{N}(0, C_{i})$ and then letting $y_{i} := G_{i}(u_{i}) + 10^{-4}\eta_{i}$, where $\eta_{i} \sim \mathcal{N}(0, I_{30 \times 30})$ i.i.d.

To minimise the objective function given by \ref{nonlinear_obj}, we apply the following variants of algorithm \ref{alg:adaptive ensemble selection}: \textbf{greedy}, \textbf{dom}, \textbf{stand}, \textbf{greedy\_r} and \textbf{dom\_r} (see Subsection \ref{subsubsection:set-up}). For each configuration of parameters J and $\beta$, 25 numerical experiments are run ($n_{\scalebox{0.75}{\mbox{exp}}}=25$). For each numerical experiment, initial conditions are generated and each of the above variants is run with these same initial conditions. All variants of EKI are run for a total time of $T=100$ and the resampling occurs at times $t = \frac{100}{3}, \frac{200}{3}$ for those variants that resample.

\subsubsection{Experiment Results}
\def\arraystretch{1.5}
\begin{table}[H]
\caption{Results for increasing J, the number of particles, in the various EKI strategies with the regularised algebraic forward operator given by Equation \eqref{eq:algebraic forward operator}. $\beta = 10^{-4}$ for all experiments.} 
\label{table:algebraic, J}
\begin{tabular}{|p{0.1\textwidth}|p{0.11\textwidth}|p{0.11\textwidth}|p{0.11\textwidth}|p{0.11\textwidth}|p{0.11\textwidth}|p{0.11\textwidth}|p{0.11\textwidth}|}
\hline
\multicolumn{6}{|c|}{\bf{Mean Residual Ratio}}\\
\hline
 &  \bf{\small J=2} & \bf{\small J=4} &  \bf{\small J=6} & \bf{\small J=8} & \bf{\small J=10} \\
\hline
$\boldsymbol{\Large \frac{r^{\mbox{\tiny min}}}{r^{\mbox{\tiny greedy\_r}}}}$ & \colorbox{lightgray}{0.0222} & \colorbox{lightgray}{0.0443} & \colorbox{lightgray}{0.115} & \colorbox{lightgray}{0.22} & \colorbox{lightgray}{0.309} \\
\hline
$\boldsymbol{\Large \frac{r^{\mbox{\tiny min}}}{r^{\mbox{\tiny greedy}}}}$ & 0.00498 & 0.00898 & 0.014 & 0.0228 & 0.0303   \\
\hline
$\boldsymbol{\Large \frac{r^{\mbox{\tiny min}}}{r^{\mbox{\tiny dom\_r}}}}$ & 0.00375 & 0.00885 & 0.0151 & 0.0288 & 0.0337 \\
\hline
$\boldsymbol{\Large \frac{r^{\mbox{\tiny min}}}{r^{\mbox{\tiny dom}}}}$ & 0.00306 & 0.00581 & 0.00872 & 0.0116 &   0.0162 \\
\hline
$\boldsymbol{\Large \frac{r^{\mbox{\tiny min}}}{r^{\mbox{\tiny stand}}}}$ & $1.22 {\times} 10^{-4}$ & $1.09 {\times} 10^{-4}$ & $1.33 {\times} 10^{-4}$ & $1.98 {\times} 10^{-4}$ & $1.87 {\times} 10^{-4}$ \\
\hline
\end{tabular}
\end{table}

\def\arraystretch{1.5}
\begin{table}[H]
\caption{Results for increasing $\beta$, the relative weight of the prior, in the various EKI strategies with the regularised algebraic forward operator given by Equation \eqref{eq:algebraic forward operator}. $J = 5$ for all experiments.}
\label{table:algebraic, beta}
\begin{tabular}{|p{0.1\textwidth}|p{0.11\textwidth}|p{0.11\textwidth}|p{0.11\textwidth}|p{0.11\textwidth}|p{0.11\textwidth}|p{0.11\textwidth}|p{0.11\textwidth}|}
\hline
\multicolumn{7}{|c|}{\bf{Mean Residual Ratio}}\\
\hline
  & $\boldsymbol{\beta = 10^{-5}}$ & $\boldsymbol{\beta = 10^{-4}}$ & $\boldsymbol{\beta = 10^{-3}}$ & $\boldsymbol{\beta = 10^{-2}}$ & $\boldsymbol{\beta = 10^{-1}}$ & $\boldsymbol{\beta = 10^{0}}$  \\
\hline
$\boldsymbol{\Large \frac{r^{\mbox{\tiny min}}}{r^{\mbox{\tiny greedy\_r}}}}$ & \colorbox{lightgray}{0.0110} & \colorbox{lightgray}{0.0799}  & \colorbox{lightgray}{0.379} & \colorbox{lightgray}{0.749} & \colorbox{lightgray}{0.97} & \colorbox{lightgray}{0.958} \\
\hline
$\boldsymbol{\Large \frac{r^{\mbox{\tiny min}}}{r^{\mbox{\tiny greedy}}}}$ & 0.00101 & 0.0116 & 0.0973  & 0.335 & 0.765 & 0.885 \\
\hline
$\boldsymbol{\Large \frac{r^{\mbox{\tiny min}}}{r^{\mbox{\tiny dom\_r}}}}$ & 0.00133 & 0.0131 & 0.0915 & 0.393 & 0.816 & 0.918\\
\hline
$\boldsymbol{\Large \frac{r^{\mbox{\tiny min}}}{r^{\mbox{\tiny dom}}}}$ & $6.85 {\times} 10^{-4}$  & 0.00805 & 0.064 & 0.283 & 0.71 & 0.875\\
\hline
$\boldsymbol{\Large \frac{r^{\mbox{\tiny min}}}{r^{\mbox{\tiny stand}}}}$ & $1.22 {\times} 10^{-5}$ & $1.47 {\times} 10^{-4}$ & 0.00139 & 0.0214 & 0.271 & 0.550\\
\hline
\end{tabular}
\end{table}
\subsubsection{Discussion of Experiment Results}
Table \ref{table:algebraic, J} shows that, for all variants of the algorithm, there is the expected monotonic decrease in the long-term value of the objective function as the number of particles increases. (This is with the exception of $r^{\scalebox{0.5}{\mbox{stand}}}$ between the values of $J = 2$ and $J = 4$, and $J = 8$ and $J = 10$. These algorithms are derived heuristically and are randomly initialised, so deviations from the expected monotonic decrease should not be too surprising.)

Table \ref{table:algebraic, beta} shows that, for all variants of the algorithm, there is the expected monotonic decrease in the long-term  value of the objective function as $\beta$ increases. This is as expected, since a larger $\beta$ can be interpreted as making the forward operator `more linear' (by putting more weight on the a linear components of the forward operator) and a monotonic decrease in the long-term value of the objective function is mathematically expected in the linear setting (see Section \ref{subsubsec: discussion, linear}).

Likewise, across all values of J and $\beta$, the relative performances of the variants of EKI are as expected: $r^{\scalebox{0.5} {\mbox{greedy\_r}}} \le r^{\scalebox{0.5}{\mbox{greedy}}}$,  $r^{\scalebox{0.5}{\mbox{dom\_r}}} \le r^{\scalebox{0.5}{\mbox{dom}}}$,  $r^{\scalebox{0.5}{\mbox{greedy\_r}}} \le r^{\scalebox{0.5}{\mbox{dom\_r}}}$, $r^{\scalebox{0.5}{\mbox{greedy}}} \le r^{\scalebox{0.5}{\mbox{dom}}}$ and $r^{\scalebox{0.5}{\mbox{dom}}} \le r^{\scalebox{0.5}{\mbox{stand}}}$ suggest that resampling, selecting eigenvectors according to Algorithm \ref{alg:selectJ} and selecting B `optimally' - i.e. according to Corollary \ref{cor:double optimal B condition} - all improve performance.

\subsubsection{2D Darcy Flow}
We consider the following elliptic PDE:
\begin{equation}\label{eq:darcy}
\begin{cases}
- \nabla \cdot (\exp(u) \nabla p) = f, & \mbox{in } D\\
p = 0, & \mbox{on } \partial D,\
\end{cases}
\end{equation}
with $D = (0, 1)^{2}$. We seek to recover the unknown diffusion coefficient $u^{\dagger} \in C^{1}(D) = X$, given an observation of the solution $p \in H^{2}(D) \cap H_0^{1}(D) =: V$ (see \cite[Chapter 5]{evans10} for details on the Sobolev spaces $H^{2}(D)$ and $H_0^{1}(D)$). Furthermore, we assume that the scalar field $f \in \mathbb{R}$ is known.
The observations are given by
\begin{equation*}
    y = \mathcal{O}(p) + \eta,
\end{equation*}
where $\mathcal{O}(p) : V \rightarrow \mathbb{R}^{K}$ is the observation operator, which considers K randomly chosen points in X, i.e. $\mathcal{O}(p) = (p(x_{1}), \ldots, p(x_{K}))$. Finally, $\eta$ denotes the noise on our data and is assumed  $\eta \sim \mathcal{N}(0, \Gamma)$ for some symmetric, non-negative $\Gamma \in \R^{k \times k}$. Then the inverse problem is given by
\begin{equation*}
    y =  G(u) + \eta,
\end{equation*}
where $G = \mathcal{O} \circ \mathcal{G}$ and $\mathcal{G} : X \rightarrow \mathbb{R}^{K}$ denotes the solution operator of \eqref{eq:darcy}. We solve the PDE on a uniform mesh with $N_{\scalebox{0.5}{\mbox{FEM}}}^2$ points $[\frac{0}{N_{\scalebox{0.5}{\mbox{FEM}}}}, \frac{1}{N_{\scalebox{0.5}{\mbox{FEM}}}}, \dots , \frac{N_{\scalebox{0.5}{\mbox{FEM}}}-1}{N_{\scalebox{0.5}{\mbox{FEM}}}}]^2$
using a finite element method (FEM) with continuous, piecewise linear finite element basis functions. We model the prior distribution as a truncated Karhunen-Loève expansion with $N_{\scalebox{0.5}{\mbox{KL}}}^{2}$ terms:
\begin{equation*}
u(x, \omega) = \mu + \sum_{k, l = 1}^{N_{\scalebox{0.5}{\mbox{KL}}}} \lambda_{k,l}^{\frac{1}{2}}e_{k,l}(x) \xi_{k,l}(\omega),
\end{equation*}
where $\mu \in \R^{N_{\scalebox{0.5}{\mbox{FEM}}}^{2}}$, $\lambda_{k,l} = (\pi^{2}(k^{2} + l^{2}) + \tau^{2})^{-\alpha}$, $e_{k, l}(x)=\frac{2}{N_{\scalebox{0.5}{\mbox{FEM}}}}\cos(2 \pi x_{1} k)\cos(2 \pi x_{2} l) \in \R^{N_{\scalebox{0.5}{\mbox{FEM}}}^{2}} $, $\xi_{k, l} \sim \mathcal{N}(0, 1)$ are i.i.d. and $(k,l)\in [1,\ldots, N_{\scalebox{0.5}{\mbox{FEM}}}]^{2}$.  Due to the FEM discretisation, $u(\vec{x}, \omega), e_{k, l }(\vec{x}), \vec{x} \in \mathbb{R}^{N_{\scalebox{0.5}{\mbox{FEM}}} \times N_{\scalebox{0.5}{\mbox{FEM}}}}$, where $\vec{x} \in \mathbb{R}^{N_{\scalebox{0.5}{\mbox{FEM}}} \times N_{\scalebox{0.5}{\mbox{FEM}}}}$ is a vector of the FEM meshpoints.

\begin{align*}
u(\vec{x}, \omega) 
  &= \mu + [e_{1, 1}(\vec{x}), \ldots , e_{N_{\scalebox{0.5}{\mbox{KL}}}, N_{\scalebox{0.5}{\mbox{KL}}}}(\vec{x})]\mbox{diag}(\lambda_{1,1}^{\frac{1}{2}}, \ldots, \lambda_{N_{\scalebox{0.5}{\mbox{KL}}}, N_{\scalebox{0.5}{\mbox{KL}}}}^{\frac{1}{2}})\xi(\omega)\\
  &=: \mu + E(\vec{x}) \Lambda^{\frac{1}{2}} \xi(\omega),
\end{align*}
where $E(\vec{x}) \in \mathbb{R}^{N_{\scalebox{0.5}{\mbox{FEM}}}^{2} \times N_{\scalebox{0.5}{\mbox{KL}}}^{2}}$, $\Lambda \in \R^{N_{\scalebox{0.5}{\mbox{KL}}}^{2} \times N_{\scalebox{0.5}{\mbox{KL}}}^{2}}$ and $\xi \sim \mathcal{N}(0, I_{N_{\scalebox{0.5}{\mbox{KL}}}}^{2})$. (The columns of $E(\vec{x})$ and $\Lambda$ are in lexicographic order with respect to the indices of $e_{k,l}(\vec{x})$ and $\lambda_{k, l}$ respectively.) We note that $E(\vec{x})^{T}E(\vec{x}) = I_{N_{\scalebox{0.5}{\mbox{KL}}}}^{2}$ and $R_{u} := Cov(u) = E\Lambda E^{T}$. We then define the un-whitened variable $w := E(\vec{x})^{T}(u - \mu)  = \Lambda^{\frac{1}{2}} \xi \in \mathbb{R}^{N_{\scalebox{0.5}{\mbox{KL}}}^{2}}$ so that $w \sim \mathcal{N}(0, \Lambda)$. Further we define the forward operator $\tilde{G} : \mathbb{R}^{N_{\scalebox{0.5}{\mbox{KL}}}^{2}} \rightarrow \mathbb{R}^{K}$ such that $\tilde{G}(w) = G (\mu+Ew)$. Then
\begin{equation}\label{darcy_objective}
\frac{1}{2}\norma{ G(u) - y}^{2} + \frac{\beta}{2} \norma{u}^{2}_{R_{u}}
= \frac{1}{2}\norma{\tilde{G}(w) - y}^{2} + \frac{\beta}{2} \norma{w}_{\Lambda}^{2}.\\
\end{equation}

To minimise this objective function, we apply EKI to the variable w rather than to u, i.e. we run EKI methods with forward operator $\begin{pmatrix}\tilde{G}(w)^{T} & \sqrt{\beta}(\Lambda^{-\frac{1}{2}}w)^{T} \end{pmatrix}^{T}$ and data $\begin{pmatrix} y^{T} & 0_{N_{\scalebox{0.5}{\mbox{KL}}}}^{T}\end{pmatrix}^{T}$.

To minimise the objective function given by \ref{darcy_objective}, we apply the following variants of algorithm \ref{alg:adaptive ensemble selection}: \textbf{greedy}, \textbf{dom}, \textbf{stand}, \textbf{greedy\_r} and \textbf{dom\_r} (see Subsection \ref{subsubsection:set-up}). For each configuration of parameters J and $\beta$, 10 numerical experiments are run ($n_{\scalebox{0.75}{\mbox{exp}}}=10$). For each numerical experiment, initial conditions are generated and each of the above variants is run with these same initial conditions. All variants of EKI are run for a total time of $T=1$ and the resampling occurs at times $t = \frac{1}{3}, \frac{2}{3}$ for those variants that resample. We take $N_{\scalebox{0.5}{\mbox{FEM}}} = 64$, $N_{\scalebox{0.5}{\mbox{KL}}} = 7$, $K = 30$, $\Gamma = 10I_{30}$, $\mu = -9$, $\alpha = 1$, $\tau = 0.01$ and $f = 1$ in Equation \eqref{eq:darcy}.

\subsubsection{Experiment Results}
\begin{table}[H]
\caption{Results for increasing the J, the number of particles, in the various EKI strategies, where the forward operator is the regularised observed solution operator to the Darcy flow problem given by PDE \eqref{eq:darcy}. In these experiments $\beta = 10^{-5}$.}
\label{table:darcy, J}
\begin{tabular}{|p{0.1\textwidth}|p{0.11\textwidth}|p{0.11\textwidth}|p{0.11\textwidth}|p{0.11\textwidth}|p{0.11\textwidth}|p{0.11\textwidth}|p{0.11\textwidth}|}
\hline
\multicolumn{6}{|c|}{\bf{Mean Residual Ratio}}\\
\hline
 &  \bf{\small J=2} & \bf{\small J=4} &  \bf{\small J=6} & \bf{\small J=8} & \bf{\small J=10} \\
\hline
$\boldsymbol{\Large \frac{r^{\mbox{\tiny min}}}{r^{\mbox{\tiny greedy\_r}}}}$ & \colorbox{lightgray}{0.339}  & \colorbox{lightgray}{0.484} & \colorbox{lightgray}{0.508} & \colorbox{lightgray}{0.621} & \colorbox{lightgray}{0.676}  \\
\hline
$\boldsymbol{\Large \frac{r^{\mbox{\tiny min}}}{r^{\mbox{\tiny greedy}}}}$ & 0.255 & 0.356 & 0.364 & 0.399 &  0.472 \\
\hline
$\boldsymbol{\Large \frac{r^{\mbox{\tiny min}}}{r^{\mbox{\tiny dom\_r}}}}$ & 0.203 & 0.287 & 0.275 & 0.324 & 0.376 \\
\hline
$\boldsymbol{\Large \frac{r^{\mbox{\tiny min}}}{r^{\mbox{\tiny dom}}}}$ & 0.203 & 0.287 & 0.275 & 0.323 & 0.322\\
\hline
$\boldsymbol{\Large \frac{r^{\mbox{\tiny min}}}{r^{\mbox{\tiny stand}}}}$ & 0.0067 & 0.0311 & 0.0751 & 0.183 &  0.272 \\
\hline
\end{tabular}
\end{table}

\noindent
\def\arraystretch{1.5}
\begin{table}[H]
\caption{Results for increasing $\beta$, the relative weight of the prior, in the various EKI strategies , where the forward operator is the regularised observed solution operator to the Darcy flow problem given by  PDE \eqref{eq:darcy}. In these experiments J = 5.}
\label{table:darcy, beta}
\begin{tabular}{|p{0.1\textwidth}|p{0.11\textwidth}|p{0.11\textwidth}|p{0.11\textwidth}|p{0.11\textwidth}|p{0.11\textwidth}|p{0.11\textwidth}|p{0.11\textwidth}|}
\hline
\multicolumn{7}{|c|}{\bf{Mean Residual Ratio}}\\
\hline
 &  $\boldsymbol{ \beta=10^{-6}}$ & $\boldsymbol{ \beta=10^{-5}}$ &  $\boldsymbol{ \beta=10^{-4}}$ & $\boldsymbol{ \beta=10^{-3}}$ & $\boldsymbol{ \beta=10^{-2}}$ & $\boldsymbol{ \beta=10^{-1}}$ \\
\hline
$\boldsymbol{\Large \frac{r^{\mbox{\tiny min}}}{r^{\mbox{\tiny greedy\_r}}}}$ & \colorbox{lightgray}{0.430} & \colorbox{lightgray}{0.479} & \colorbox{lightgray}{0.548} & \colorbox{lightgray}{0.772} & \colorbox{lightgray}{0.947} & \colorbox{lightgray}{0.987} \\
\hline
$\boldsymbol{\Large \frac{r^{\mbox{\tiny min}}}{r^{\mbox{\tiny greedy}}}}$ & 0.314 & 0.311 & 0.404 & 0.549 & 0.778 & 0.926\\
\hline
$\boldsymbol{\Large \frac{r^{\mbox{\tiny min}}}{r^{\mbox{\tiny dom\_r}}}}$ & 0.243 & 0.233 & 0.313 & 0.402 & 0.618 & 0.851\\
\hline
$\boldsymbol{\Large \frac{r^{\mbox{\tiny min}}}{r^{\mbox{\tiny dom}}}}$ & 0.243 & 0.233 & 0.313 & 0.402 & 0.617 & 0.851 \\
\hline
$\boldsymbol{\Large \frac{r^{\mbox{\tiny min}}}{r^{\mbox{\tiny stand}}}}$  & 0.0121 & 0.0495 & 0.228 & 0.341 & 0.591 & 0.810 \\
\hline
\end{tabular}
\end{table}

\subsubsection{Discussion of Experiment Results}
Table \ref{table:darcy, J} shows that, for all variants of the algorithm, there is the expected monotonic decrease in the long-term  value of the objective function as number of particles increases. (This is with the exception of $r^{\scalebox{0.5}{\mbox{dom}}}$ between the values of $J = 16$ and $J = 32$. These algorithms are derived heuristically and are randomly initialised, so deviations from the expected monotonic decrease should not be too surprising.)

Table \ref{table:darcy, beta} shows that, for all variants of the algorithm, there is the expected monotonic decrease in the long-term  value of the objective function as $\beta$ increases. This is as expected, since a larger $\beta$ can be interpreted as making the forward operator `more linear' (by putting more weight on a linear components of the forward operator) and a monotonic decrease in the long-term value of the objective function is mathematically expected in the linear setting (see Section \ref{subsubsec: discussion, linear}).

Likewise, across all values of J and $\beta$, the relative performances of the variants of EKI are as expected: $r^{\scalebox{0.5} {\mbox{greedy\_r}}} \le r^{\scalebox{0.5}{\mbox{greedy}}} \le r^{\scalebox{0.5}{\mbox{dom\_r}}} \le r^{\scalebox{0.5}{\mbox{dom}}} \le r^{\scalebox{0.5}{\mbox{stand}}}$ suggests that resampling, selecting eigenvectors according to Algorithm ref{alg:selectJ} and selecting B `optimally' - i.e. according to Corollary \ref{cor:double optimal B condition} - all improve performance.

\section{Conclusion and Outlook}\label{sec:concl}
This work represents a comprehensive study on optimising the subspace within the EKI framework to address challenges associated with poor prior knowledge and small ensemble sizes. By explicitly solving the EKI particle dynamics in the linear setting, we develop a novel greedy strategy for selecting a solution subspace and strategies for selecting the initial ensemble within this subspace. We extend these results to develop an algorithm for performing EKI with adaptive resampling that is suitable for both linear and nonlinear problems. 

Numerical experiments validate the theoretical findings, showcasing the effectiveness of the proposed methods in diverse scenarios, including high-dimensional problems and ill-posed settings. These experiments show that, in terms of accuracy, the proposed method significantly outperforms standard EKI methods and performs well even compared to more computationally expensive optimisation methods. Despite small ensemble sizes and a limited number of derivative calculations, the proposed methods effectively leverage the structure of inverse problems to yield impressive results. 

Future work could further develop the adaptive resampling in the proposed algorithm to improve its performance in highly nonlinear settings. More specifically, one could systematically investigate how to effectively choose certain parameters in the algorithm, such as ensemble sizes and resample times. Furthermore, it would be interesting to investigate alternative linearisations of the forward operator in order to make the proposed algorithm entirely derivative-free. Another direction of future research would be to combine our proposed method with machine learning techniques that use suitable embeddings to transform nonlinear problems into linear ones.

\section*{Acknowledgements}
The work of RH has been funded by Deutsche Forschungsgemeinschaft (DFG) through the grant CRC 1114 ‘Scaling Cascades in Complex Systems’ (project number 235221301, project A02 - Multiscale data and asymptotic model assimilation for atmospheric flows).  CS acknowledges support from MATH+ Project EF1-19: Machine Learning Enhanced Filtering Methods for Inverse Problems and EF1-20: Uncertainty Quantification and Design of Experiment for Data-Driven Control, funded by the Deutsche Forschungsgemeinschaft (DFG, German Research Foundation) under Germany’s Excellence Strategy—The Berlin Mathematics Research Center MATH+ (EXC-2046/1, Project ID: 390685689).

\appendix
\section{Technical results}\label{sec:appendix}

Several key quantities in the analysis of EKI for linear problems evolve according to the dynamics described in the following theorem. As noted by \cite{bungert2021complete}, the mean-field dynamics of EKI evolve according to equations below with $\alpha = 1$, while the leading term in the average sample equation for stochastic linear EKI with J particles evolves according to the equations below with $\alpha = \frac{J + 1}{J}$. In this paper, however, we focus solely on the case $\alpha = 2$, which characterises the particle dynamics of deterministic linear EKI.

\begin{theorem}
\label{thm:general_dynamics_thm}
Let $\alpha > 0$, $A\in \mathbb R^{k\times n}$, $x_0 \in \mathbb R^{n}$, $C_0 \in \mathbb R^{n\times n}$ be symmetric and  positive semi-definite. Further let $C_{0} = C_{L}C_{R} \in \mathbb{R}^{n \times n}$ be any desired decomposition, where $C_{L} \in \mathbb{R}^{n \times p}$, $C_{R} \in \mathbb{R}^{p \times n}$ for some $p \in \N^{+}$. Let $AC_{0}A^{T}$ be of rank $r$, let $AC_{0}A^{T} = \mathcal{U} \Sigma \mathcal{U}^{T}$ be any desired eigen-decomposition with $\mathcal{U} \in \mathbb{R}^{n \times r}$ and $\Sigma \in \mathbb{R}^{r \times r}$.
Then the initial value problem (IVP)
\begin{align}\label{eq:C-x IVP}
\frac{d}{dt}C_{t} &= -\alpha C_{t}A^{T}AC_{t}, \quad t>0,\\
\frac{d}{dt}x_{t} &= -C_{t}A^{T}(Ax_{t} - y), \quad t>0,\nonumber\\
C(0)&=C_0, \quad u(0)=u_0, \nonumber
\end{align}
has unique solution
\begin{align*}
C_{t} &= C_{L}(I_{p} + \alpha C_{R}A^{T}AC_{L}t)^{-1}C_{R},\\
x_{t} &= x_{0} + (AC_{0})^{T} \mathcal{U}\Sigma^{-1}((I_{r} + \alpha \Sigma t)^{-\frac{1}{\alpha}}- I_{r})\mathcal{U}^{T}(Ax_{0} - y).
\end{align*}
\end{theorem}
\begin{proof}
It can be easily verified that the given expression for $C_{t}$ satisfies the IVP \eqref{eq:C-x IVP}. Perhaps the trickiest part in doing so is to prove that this expression is well defined, i.e. that $(I + (\alpha t) C_{L}A^{T}AC_{R})$ is invertible.
\begin{equation*}
(I + (\alpha t) C_{R}A^{T}AC_{L})x = 0, 
\end{equation*}
implies
\begin{equation*}
0 = AC_{L}(I + (\alpha t)C_{R}A^{T}AC_{L})x  = (I + (\alpha t)AC_{0}A^{T})AC_{L}x,\\
\end{equation*}
which leads to
\begin{equation*}
AC_{L}x = 0,
\end{equation*}
and thus
\begin{equation*}
x = -(\alpha t)C_{R}A^{T}AC_{L}x = 0.
\end{equation*}
Likewise, it can be easily verified that the given expression for $x_{t}$ - in combinations with the expression for $C_{t}$ - satisfies the IVP \eqref{eq:C-x IVP}. However, we presently include a derivation of this expression, since this might prove instructive for solving similar ODEs arising from EKI. To begin with, from the first ODE in the IVP \eqref{eq:C-x IVP}, we derive the following ODE for $AC_{t}A^{T}$:

\begin{equation}\label{eq:ACAode}
\frac{d}{dt}AC_{t}A^{T} = -\alpha (AC_{t}A^{T})(AC_{t}A^{T}).
\end{equation}
If $AC_{0}A^{T} = \mathcal{U}\Sigma\mathcal{U}^{T}$ is a diagonal decomposition, then the solution to \eqref{eq:ACAode} is given by
\begin{align}\label{eq:sln ACA ode}
AC_{t}A^{T} &= \mathcal{U} \Sigma_{t}\mathcal{U}^{T}, \nonumber\\
\Sigma_{t} &= \Sigma(I + \alpha \Sigma t)^{-1}.
\end{align}
Next, from the first ODE in the IVP (\ref{eq:C-x IVP}), we derive the ODE
\begin{equation*}
\frac{d}{dt}(AC_{t}) = -\alpha (AC_{t}A^{T})(AC_{t}),\\
\end{equation*}
which we solve using our expression for $AC_{t}A^{T}$ given by Equation (\ref{eq:sln ACA ode})  
\begin{align*}
AC_{t} &= \exp{(-\alpha\int_{0}^{t}\mathcal{U}\Sigma_{s}\mathcal{U}^{T}ds)}AC_{0}\\
&= (I - \mathcal{U}\mathcal{U}^{T} + \mathcal{U}(I + \alpha\Sigma t)^{-\frac{1}{\alpha}}\mathcal{U}^{T})AC_{0}.
\end{align*}
Similarly, from the second ODE in the IVP \eqref{eq:C-x IVP}, we derive the expression  
\begin{equation*}
\frac{d}{dt}(Ax_{t} - y) = - (AC_{t}A^{T})(Ax_{t} - y),
\end{equation*}
which we solve using our expression for $AC_{t}A^{T}$ given by Equation \eqref{eq:sln ACA ode}
\begin{align*}
Ax_{t} - y &= \exp{(-\int_{0}^{t}AC_{s}A^{T}ds)}(Ax_{0} - y)\\
&= (I - \mathcal{U}\mathcal{U}^{T} + \mathcal{U}(I + \alpha \Sigma t)^{-1}\mathcal{U}^{T})(Ax_{0} - y).
\end{align*}
Finally, 
\begin{align*}
x_{t} &= x_{0} + \int_{0}^{t}\frac{d}{ds}x_{s}ds\\
&= x_{0} - \int_{0}^{t}(AC_{s})^{T}(Ax_{s} - y)ds\\
&= x_{0} + (AC_{0})^{T} \mathcal{U}\Sigma^{-1}((I + \alpha \Sigma t)^{-\frac{1}{\alpha}}- I)\mathcal{U}^{T}(Ax_{0} - y).
\end{align*}
\end{proof}

\begin{lemma} \label{lemma:EV_UV}
Recall that $\mathcal{E}_{0} = [u^{(1)}_{0} - \bar{u}(0), \ldots,  u^{(J)}_{0} - \bar{u}(0)] = U_{0}(I_{J} - \frac{1}{J}1_{J}1_{J}^{T})$, where $1_{J} \in \R^{J}$ is a vector of ones. Let $A\mathcal{E}_{0} = \sqrt{J}\mathcal{U}\Sigma \mathcal{V}^{T}$ be a compact singular value decomposition ($\mathcal{U} \in \mathbb{R}^{m \times r}$, $\Sigma \in \mathbb{R}^{r \times r}$, $\mathcal{V} \in \mathbb{R}^{J \times r}$, where $\Sigma \succ 0$). Then 
$\mathcal{E}_{0}\mathcal{V} = U_{0}\mathcal{V}$.
\begin{proof}
There holds
\begin{align*}
\mathcal{V}^{T}1_{J} &= \frac{1}{\sqrt{J}}\Sigma^{-1}\mathcal{U}^{T}(\sqrt{J}\mathcal{U}\Sigma\mathcal{V}^{T})1_{J}\\
 &= \frac{1}{\sqrt{J}}\Sigma^{-1}\mathcal{U}^{T}A\mathcal{E}_{0}1_{J}\\
 &= \frac{1}{\sqrt{J}}\Sigma^{-1}\mathcal{U}^{T}AU_{0}(I_{J} - \frac{1}{J}1_{J}1_{J}^{T})1_{J}\\
&=0.
\end{align*}
Then
\begin{align*}
\mathcal{E}_{0}\mathcal{V} &= U_{0}(I_{J} - \frac{1}{J}1_{J}1_{J}^{T})\mathcal{V}\\
&= U_{0}\mathcal{V}.
\end{align*}
\end{proof}
\end{lemma}

\begin{lemma}\label{Im(AE)_perp}
Let $R \in \R^{n \times n}$ have eigen-decomposition $R = \mathcal{V}\Lambda \mathcal{V}^{T}$, where $\mathcal{V}, \Lambda \in \R^{n \times n}$. Let $\mathcal{J} \subset \{1, \ldots, n\}$ be an index set of size $J\le n$. If $\mathcal{V} = [v_{1}, \ldots, v_{n}]$, define $\mathcal{V}_{\mathcal{J}} := [v_{\mathcal{J}_{1}}, \ldots, v_{\mathcal{J}_{J}}]$, $\mathcal{V}_{\mathcal{J}^{c}} := [v_{\mathcal{J}^{c}_{1}}, \ldots, v_{\mathcal{J}^{c}_{J}}]$, $\Lambda_{\mathcal{J}} := \mbox{diag}(\{\lambda_{\mathcal{J}_{1}}, \ldots, \lambda_{\mathcal{J}_{J}} \})$, $R_{\mathcal{J}}^{\frac{1}{2}} := \mathcal{V}_{\mathcal{J}} \Lambda_{\mathcal{J}}^{\frac{1}{2}}\mathcal{V}_{\mathcal{J}}^{T}$ and $(R_{\mathcal{J}}^{\dagger})^{\frac{1}{2}} := \mathcal{V}_{\mathcal{J}} \Lambda_{\mathcal{J}}^{-\frac{1}{2}}\mathcal{V}_{\mathcal{J}}^{T}$. $U_{0} := \mathcal{V}_{\mathcal{J}}B$, where $B \in \R^{J \times J}$ is invertible. $C_{0} := \frac{1}{J}U_{0}(I_{J} - \frac{1}{J}1_{J}1_{J}^{T})U_{0}^{T}$, where $1_{J} \in \R^{J}$ is a vector of ones. $A \in \mathbb{R}^{m \times n}$ and $\tilde{A} := \begin{pmatrix} A \\ (R_{\mathcal{J}}^{\dagger})^{\frac{1}{2}} \end{pmatrix} \in \mathbb{R}^{(m + n) \times n}$.
Then 
\begin{equation*}
{\mbox{ker}}(\tilde{A}C_{0}\tilde{A}^{T}) = \mbox{Im}\begin{pmatrix} I_{m \times m} & 0_{m \times 1} & 0_{m \times (n -J)} \\ -R_{\mathcal{J}}^{\frac{1}{2}}A^{T} &  R_{\mathcal{J}}^{\frac{1}{2}}\mathcal{V}_{\mathcal{J}}B^{-T}1_{J} & \mathcal{V}_{\mathcal{J}^{c}} \end{pmatrix}.
\end{equation*}
Furthermore, the columns of the matrix on the right-hand side are linearly independent.
\begin{proof}
Take
\begin{align*}
p \in \mbox{ker}(\tilde{A}C_{0}\tilde{A}^{T}).
\end{align*}
Then
\begin{align*}
0 = (I_{J} - \frac{1}{J}1_{J}1_{J}^{T})B^{T}V_{\mathcal{J}}^{T}\tilde{A}^{T}p,
\end{align*}
and thus
\begin{align*}
\tilde{A}^{T}p = \alpha \mathcal{V}_{\mathcal{J}}B^{-T}1_{J} + x,
\end{align*}
for some $x  \in \mbox{Im}(\mathcal{V}_{ \mathcal{J}^{c}})$.

Writing $p = \begin{pmatrix} r \\ s\end{pmatrix}$ for $r \in \mathbb{R}^{m}$ and $s \in \mathbb{R}^{n}$, we get
\begin{align*}
A^{T}r + (R_{\mathcal{J}}^{\dagger})^{\frac{1}{2}}s  = \tilde{A}^{T}p =  \alpha  \mathcal{V}_{\mathcal{J}}B^{-T}1_{J}+ x.
\end{align*}

Rearranging,
\begin{equation*}
\mathcal{V}_{\mathcal{J}}\mathcal{V}_{\mathcal{J}}^{T}s = -R_{\mathcal{J}}^{\frac12}A^{T}r + \alpha \mathcal{V}_{\mathcal{J}}\Lambda_{\mathcal{J}}^{\frac{1}{2}}B^{-T}1_{J},
\end{equation*}
and therefore
\begin{equation*}
p =  \begin{pmatrix} I_{m} \\ -R_{\mathcal{J}}^{\frac{1}{2}}A^{T}\end{pmatrix}r +  \alpha \begin{pmatrix} 0_{m \times 1} \\ \mathcal{V}_{\mathcal{J}}\Lambda_{\mathcal{J}}^{\frac{1}{2}}B^{-T}1_{J}\end{pmatrix} + \begin{pmatrix} 0_{m \times 1} \\ \mathcal{V}_{\mathcal{J}^{c}}\mathcal{V}_{\mathcal{J}^{c}}^{T}s\end{pmatrix}.
\end{equation*}
It follows that $\mbox{ker}(\tilde{A}C_{0}\tilde{A}^{T})$ is spanned by the columns of the following matrix:

\begin{equation*}
\begin{pmatrix} I_{m} & 0_{m \times 1} & 0_{m \times (n -J)} \\ -R_{\mathcal{J}}^{\frac{1}{2}}A^{T} &  R_{\mathcal{J}}^{\frac{1}{2}}\mathcal{V}_{\mathcal{J}}B^{-T}1_{J} & \mathcal{V}_{\mathcal{J}^{c}} \end{pmatrix} \in \mathbb{R}^{(m+n)\times(m + n - J + 1)}.
\end{equation*}
It is clear that the columns of this matrix are linearly independent. 
\end{proof}
\end{lemma}

\begin{lemma}\label{lemma:minimal trace diagonal}
Let $x, y \in \mbox{R}^{n}$, $\Sigma \in \R^{n \times n}$ be diagonal with $\Sigma \succ 0 \mbox{ and } 1 = x^{T} \Sigma^{-1} y$. Define $\alpha \in \mbox{R}^{n}$ by $\alpha_{i} := x_{i}y_{i}$. Then 
\begin{equation*}
    \mbox{trace}(\Sigma^{2}) \ge (\norma{\alpha}_{\frac{2}{3}})^{2}.
\end{equation*}
Furthermore, this lower bound is achieved iff
\begin{equation*}
\forall k: \sigma_{k} = \alpha_{k}^{\frac{1}{3}}(\norma{\alpha}_{\frac{2}{3}})^{\frac{2}{3}}.
\end{equation*}
\begin{proof}
There holds
\begin{equation*}
\{ \Sigma : \Sigma \succ 0, \quad x^{T}\Sigma^{-1} y = 1 \}  = \{ (x^{T}\Sigma^{-1} y)\Sigma : \Sigma \succ 0 , \quad x^{T}\Sigma^{-1} y \ne 0 \}.
\end{equation*}
With
$
\tilde{\Sigma} := (x^{T}\Sigma^{-1} y)\Sigma
$ we get
\begin{align*}
\frac{\partial}{\partial \sigma_{k}}\mbox{trace}(\tilde{\Sigma}^{T}\tilde{\Sigma}) &= \frac{\partial}{\partial \sigma_{k}}\left(\sum_{i}\frac{\alpha_{i}}{\sigma_{i}}\right)^{2}\sum_{j}\sigma_{j}^{2}\\
&= -2\frac{\alpha_{k}}{\sigma_{k}^{2}}\left(\sum_{i}\frac{\alpha_{i}}{\sigma_{i}}\right)\sum_{j}\sigma_{j}^{2} + 2 \sigma_{k} \left(\sum_{i}\frac{\alpha_{i}}{\sigma_{i}}\right)^{2}.
\end{align*}
Setting
\begin{equation*}
\frac{\partial}{\partial \sigma_{k}}\mbox{trace}(\tilde{\Sigma}^{T}\tilde{\Sigma}) =0,
\end{equation*}
and re-arranging, we get 
\begin{equation*}
    \sigma_{k} = \gamma \alpha_{k}^{\frac{1}{3}},
\end{equation*}
for some $\gamma \ne 0$. This yields
\begin{align*}
\tilde{\sigma}_{k} &= \left(\sum_{i}\frac{\alpha_{i}}{\sigma_{i}}\right) \sigma_{k} = \left(\sum_{i}\frac{\alpha_{i}}{\gamma \alpha_{i}^{\frac{1}{3}}}\right) \gamma \alpha_{k}^{\frac{1}{3}} =\alpha_{k}^{\frac{1}{3}}\sum_{i}\alpha_{i}^{\frac{2}{3}}.
\end{align*}
\end{proof}
\end{lemma}

\begin{lemma}\label{lemma:minimal trace}
Let $x, y \in \mathbb{R}^{J}$ and $U, V \in \mathbb{R}^{J \times J}$ be orthogonal matrices such that
\begin{align*}
V^{T}x/\norma{x}_{2} &= 1_{J} / \sqrt{J},\\
U^{T}y/\norma{y}_{2} &= 1_{J} /\sqrt{J},
\end{align*}
where $1_{J} \in \mathbb{R}^{J}$ is a vector of ones. Define
\begin{equation*}
B := \norma{x}_{2}\norma{y}_{2}UV^{T}.
\end{equation*}
Then
\begin{align*}
1 &= x^{T}B^{-1}y,\\
\mbox{trace}(B^{T}B) &= J \norma{x}_{2}^{2}\norma{y}_{2}^{2}.
\end{align*}
\begin{proof}
Let $\tilde{B} \in \mathbb{R}^{J \times J}$  be invertible with singular value decomposition $\tilde{U}\tilde{\Sigma}\tilde{V}^{T}$ such that $1 = x^{T} \tilde{B}^{-1} y = (\tilde{V}^{T}x)^{T}\tilde{\Sigma}^{-1}(\tilde{U}^{T}y)$. Suppose we are choosing $\tilde{\Sigma}$ to minimise $\mbox{trace}(\tilde{B}^{T}\tilde{B}) = \mbox{trace}(\tilde{\Sigma}^{T}\tilde{\Sigma})$. According to Lemma \ref{lemma:minimal trace diagonal}, we should set
\begin{equation*}
\tilde{\sigma}_{j} = (\tilde{V}^{T}x)_{i}^{\frac{1}{3}}(\tilde{U}^{T}y)_{i}^{\frac{1}{3}}(\norma{(\tilde{V}^{T}x)(\tilde{U}^{T}y)}_{\frac{2}{3}})^{\frac{3}{2}}.
\end{equation*}
In this case
\begin{align*}
\mbox{trace}(\tilde{B}^{T}\tilde{B}) &= (\norma{(\tilde{V}^{T}x)(\tilde{U}^{T}y)}_{\frac{2}{3}})^{2}\\
&\le (\norma{(\tilde{V}^{T}x)}_{\frac{4}{3}})^{2} (\norma{(\tilde{U}^{T}y)}_{\frac{4}{3}})^{2}.
\end{align*}
Setting $\tilde{U} = U$ and $\tilde{V} = V$ - where U and V are defined above - minimises the right-hand side of this inequality. In this case
\begin{align*}
\tilde{\sigma}_{j} &= \norma{x}_{2}\norma{y}_{2},\\
\tilde{B} &= \norma{x}_{2}\norma{y}_{2}UV^{T}.
\end{align*}
\end{proof}
\end{lemma}

\section{Proofs of the Main Results}\label{sec:AppB}
\noindent
\title{Proof of \textbf{Theorem \ref{thm:main thm}}}
\begin{proof}
\phantomsection\label{proof:main thm}
Let $u \in \mbox{span}\{v_{\mathcal{J}_i}\}_{i \le J}$. Then
\begin{align*}
\norma{Au - y}^{2} + \norma{u - \mu}_{R}^{2} &= \norma{Au -y}^{2} + \norma{u - \mu}_{R_{\mathcal{J}}}^{2} + \norma{u - \mu}_{R_{\mathcal{J}^{c}}}^{2}\\
&= \norma{\begin{pmatrix}A \\ (R_{\mathcal{J}}^{\dagger})^{\frac{1}{2}}\end{pmatrix}u - \begin{pmatrix}y\\ (R_{\mathcal{J}}^{\dagger})^{\frac{1}{2}}\mu \end{pmatrix}}^{2} + \norma{\mu}_{R_{\mathcal{J}^{c}}}^{2}\\
&= \norma{\tilde{A}u - \tilde{y}}^{2} + \norma{\mu}_{R_{\mathcal{J}^{c}}}^{2},
\end{align*}

where $\tilde{A} := \begin{pmatrix}A \\ (R_{\mathcal{J}}^{\dagger})^{\frac{1}{2}}\end{pmatrix} \in \mathbb{R}^{(m + n) \times n}$ and $\tilde{y} := \begin{pmatrix}y\\ (R_{\mathcal{J}}^{\dagger})^{\frac{1}{2}}\mu \end{pmatrix} \in \mathbb{R}^{m + n}$. 

Since  $[u_{\mathcal{J}, B}^{(1)}(0), \ldots , u_{\mathcal{J}, B}^{(J)}(0)] = \mathcal{V}_{\mathcal{J}}B$, the subspace property (Lemma \ref{lem:invariant_subspace}) means that $u_{\mathcal{J}, B}^{(i)}(t) \in \mbox{span}\{v_{\mathcal{J}_i}\}_{i \le J}$  for all $t \ge 0$. Furthermore, Since $\mbox{span}\{v_{\mathcal{J}_i}\}_{i \le J} \subseteq \R^{n}$ is a closed subspace, $\lim_{t \rightarrow \infty}u_{\mathcal{J}, B}^{(i)}(t) \in \mbox{span}\{v_{\mathcal{J}_i}\}_{i \le J}$ as well. Due to Corollary \ref{ensemble_collapse}, all particles converge to the same point in projection space, i.e. $\lim_{t \rightarrow \infty} \tilde{A}u_{\mathcal{J}, B}^{(i)}(t) = \lim_{t \rightarrow \infty} \tilde{A}\bar{u}_{\mathcal{J}, B}^{(i)}(t)$. As such, we are interested in calculating

\begin{equation*}
\lim_{t \rightarrow \infty}\norma{A\bar{u}_{\mathcal{J}, B}(t) - y}^{2} + \norma{\bar{u}_{\mathcal{J}, B}(t) - \mu}_{R}^{2} = \lim_{t \rightarrow \infty}\norma{\tilde{A}\bar{u}_{\mathcal{J}, B}(t) - \tilde{y}}^{2} + \norma{\mu}_{R_{\mathcal{J}^{c}}^{\dagger}}^{2}.
\end{equation*}

According to Corollary \ref{terminal_residual}, $\lim_{t \rightarrow \infty}\tilde{A}\bar{u}_{\mathcal{J}, B}(t) - \tilde{y} = Proj_{\mbox{ker}(\tilde{A}C_{0}\tilde{A}^{T})}(\tilde{A}\bar{u}_{\mathcal{J}, B}(0) - \tilde{y})$. From Lemma \ref{Im(AE)_perp} we have that $\mbox{ker}(\tilde{A}C_{0}\tilde{A}^{T}) = \mbox{Im}(P)$, where $P = \begin{pmatrix} I_{m}  & 0_{m \times 1} & 0_{m \times(n -J)} \\ -R^{\frac{1}{2}}_{\mathcal{J}}A^{T} &  \beta & \mathcal{V}_{\mathcal{J}^{c}}\end{pmatrix}$ and $\beta := R^{\frac{1}{2}}_{\mathcal{J}}\mathcal{V}_{\mathcal{J}}B^{-T}1_{J} = \mathcal{V}_{\mathcal{J}}\Lambda_{\mathcal{J}}^{\frac{1}{2}}B^{-T}1_{J}$. Since the columns of P are linearly independent, $P^{T}P$ is invertible and $Proj_{\mbox{ker}(\tilde{A}C_{0}\tilde{A}^{T})} = P(P^{T}P)^{-1}P^{T}$. As such
\begin{align*}
\norma{Proj_{\mbox{ker}(\tilde{A}C_{0}\tilde{A}^{T})}(\tilde{A}\bar{u}_{\mathcal{J}, B}(0) - \tilde{y})}^{2} = (P^{T}(\tilde{A}\bar{u}_{\mathcal{J}, B}(0) - \tilde{y}))^{T}(P^{T}P)^{-1}P^{T}(\tilde{A}\bar{u}_{\mathcal{J}, B}(0) - \tilde{y}).
\end{align*}
We first calculate
\begin{align*}
P^{T}&(\tilde{A}\bar{u}_{\mathcal{J}, B}(0) - \tilde{y}) = P^{T}\left(\frac{1}{J} \begin{pmatrix}A\\ (R_{\mathcal{J}}^{\dagger})^{\frac{1}{2}}\end{pmatrix}\mathcal{V}_{\mathcal{J}}B1_{J} - \begin{pmatrix}y \\ (R_{\mathcal{J}}^{\dagger})^{\frac{1}{2}}\mu\end{pmatrix}\right)\\
&= \begin{pmatrix} I_{m} & -AR_{\mathcal{J}}^{\frac{1}{2}}\\ 0_{1 \times m} & \beta^{T}\\ 0_{(n -J) \times m} & \mathcal{V}_{\mathcal{J}^{c}}^{T}\end{pmatrix}\left(\frac{1}{J}\begin{pmatrix}A\mathcal{V}_{\mathcal{J}}B1_{J}\\ (R_{\mathcal{J}}^{\dagger})^{\frac{1}{2}}\mathcal{V}_{\mathcal{J}}B1_{J}\end{pmatrix}  - \begin{pmatrix}y \\ (R_{\mathcal{J}}^{\dagger})^{\frac{1}{2}}\mu\end{pmatrix}\right)\\
&= \begin{pmatrix}0_{m \times 1} \\ \frac{1}{J}\beta^{T}(R_{\mathcal{J}}^{\dagger})^{\frac{1}{2}}\mathcal{V}_{\mathcal{J}}B1_{J} \\ 0_{(n-J)\times1} \end{pmatrix} - \begin{pmatrix}y - AV_{\mathcal{J}}V_{\mathcal{J}}^{T}\mu\\ \beta^{T}(R_{\mathcal{J}}^{\dagger})^{\frac{1}{2}}\mu \\ 0_{(n-J)\times1} \end{pmatrix} = \begin{pmatrix}0_{m \times 1} \\ 1 \\0_{(n-J)\times1} \end{pmatrix} - \begin{pmatrix}y - AV_{\mathcal{J}}V_{\mathcal{J}}^{T}\mu\\ 1_{J}^{T}B^{-1}\mathcal{V}_{\mathcal{J}}^{T}\mu \\ 0_{(n-J)\times1} \end{pmatrix}\\
&= \begin{pmatrix}-(y - AV_{\mathcal{J}}V_{\mathcal{J}}^{T}\mu)\\ 1 - 1_{J}^{T}B^{-1}\mathcal{V}_{\mathcal{J}}^{T}\mu \\0_{(n-J)\times1}\end{pmatrix} =: \begin{pmatrix} -z_{\mathcal{J}}\\ \gamma \\ 0_{(n-J)\times1} \end{pmatrix}.
\end{align*}
Now we turn our attention to $(P^{T}P)^{-1}$:
\begin{align*}
(P^{T}P)^{-1} &= \begin{pmatrix}I_{m} + AR_{\mathcal{J}}A^{T} & -AR_{\mathcal{J}}^{\frac{1}{2}}\beta & 0_{m \times (n - J)}\\	                   -\beta^{T}R_{\mathcal{J}}^{\frac{1}{2}}A^{T} &  \beta^{T}\beta & 0_{1\times(n-J)} \\ 0_{(n-J)\times m}  & 0_{(n-J)\times 1} & I_{(n-J)}\end{pmatrix}^{-1}\\
&=: \begin{pmatrix} Q_{\mathcal{J}}^{-1} & -\alpha & 0_{m \times (n-J)}\\
-\alpha & \beta^{T}\beta & 0_{1 \times (n-J)}\\0_{(n-J)\times m} & 0_{(n-J)\times1} & I_{(n-J)}
\end{pmatrix}^{-1}\\
&= \frac{1}{\Delta}
\begin{pmatrix}
\Delta Q_{\mathcal{J}} + Q_{\mathcal{J}}\alpha(Q_{\mathcal{J}}\alpha)^{T} & Q_{\mathcal{J}}\alpha & 0_{m \times (n-J)} \\
(Q_{\mathcal{J}}\alpha)^{T} & 1 & 0_{1 \times (n-J)}\\
0_{(n-J)\times m} & 0_{(n-J) \times 1} & \Delta I_{(n-J)}
\end{pmatrix},
\end{align*}
where $\Delta := \beta^{T}\beta - \alpha^{T}Q_{\mathcal{J}}\alpha$. Therefore
\begin{align*}
&|Proj_{\mbox{ker}(\tilde{A}C_{0}\tilde{A}^{T})}(\tilde{A}\bar{u}_{\mathcal{J}, B}(0) - \tilde{y})|^{2} = (P^{T}(\tilde{A}\bar{u}_{\mathcal{J}, B}(0) - \tilde{y}))^{T}(P^{T}P)^{-1}P^{T}(\tilde{A}\bar{u}_{\mathcal{J}, B}(0) - \tilde{y})\\
&= \frac{1}{\Delta}\begin{pmatrix}-z_{\mathcal{J}}^{T} & \gamma & 0_{1 \times (n-J)}\end{pmatrix}
\begin{pmatrix}
\Delta Q_{\mathcal{J}} + Q_{\mathcal{J}}\alpha(Q_{\mathcal{J}}\alpha)^{T} & Q_{\mathcal{J}}\alpha & 0_{m \times (n-J)} \\
(Q_{\mathcal{J}}\alpha)^{T} & 1 & 0_{1 \times (n-J)}\\
0_{(n-J)\times m} & 0_{(n-J) \times 1} & \Delta I_{(n-J)}
\end{pmatrix}
\begin{pmatrix}-z_{\mathcal{J}} \\ \gamma \\ 0_{(n-J) \times 1}\end{pmatrix}\\
&= z_{\mathcal{J}}^{T}Q_{\mathcal{J}}z_{\mathcal{J}} + \frac{1}{\Delta}(\alpha^{T}Q_{\mathcal{J}}z_{\mathcal{J}} - \gamma)^{2}.
\end{align*}

Recall $\beta := \mathcal{V}_{\mathcal{J}}\Lambda_{\mathcal{J}}^{\frac{1}{2}}B^{-T}1_{J}$, $\alpha := AR_{\mathcal{J}}^{\frac{1}{2}}\beta = A\mathcal{V}_{\mathcal{J}}\Lambda_{\mathcal{J}}B^{-T}1_{J}$, $z_{\mathcal{J}} := y - A\mathcal{V}_{\mathcal{J}}\mathcal{V}_{\mathcal{J}}^{T}\mu$, $\gamma := 1 - 1_{J}^{T}B^{-1}\mathcal{V}_{\mathcal{J}}^{T}\mu$ and $\Delta := \beta^{T}\beta - \alpha^{T}Q_{\mathcal{J}}\alpha$. Defining $A_{\mathcal{J}} := A\mathcal{V}_{\mathcal{J}}$ for convenience, we see that

\begin{align*}
&\lim_{t \rightarrow \infty}(\norma{A\bar{u}_{\mathcal{J}, B}(t) - y}^{2} + \norma{\bar{u}_{\mathcal{J}, B}(t) - \mu}_{R}^{2}) = \lim_{t \rightarrow \infty}\norma{\tilde{A}\bar{u}_{\mathcal{J}, B}(t) - \tilde{y}}^{2} + \norma{\mu}_{R_{\mathcal{J}^{c}}^{\dagger}}^{2}\\
&= \norma{Proj_{\mbox{ker}(\tilde{A}C_{0}\tilde{A}^{T})}(\tilde{A}\bar{u}_{\mathcal{J}, B}(0) - \tilde{y})}^{2} + \norma{\mu}_{R_{\mathcal{J}^{c}}^{\dagger}}^{2}
\\
&= z_{\mathcal{J}}^{T}Q_{\mathcal{J}}z_{\mathcal{J}} + \frac{1}{\Delta}(\alpha^{T}Q_{\mathcal{J}}z_{\mathcal{J}} - \gamma)^{2} + \norma{\mu}_{R_{\mathcal{J}^{c}}^{\dagger}}^{2}\\
&= z_{\mathcal{J}}^{T}Q_{\mathcal{J}}z_{\mathcal{J}} + \frac{(1 - 1_{J}^{T}B^{-1}(\mathcal{V}_{\mathcal{J}}^{T}\mu + \Lambda_{\mathcal{J}}A_{\mathcal{J}}^{T}Q_{\mathcal{J}}z_{\mathcal{J}}))^{2}}{ 1_{J}^{T}B^{-1}\Lambda_{\mathcal{J}}B^{-T}1_{J} - (A_{\mathcal{J}}\Lambda_{\mathcal{J}}B^{-T}1_{J})^{T}Q_{\mathcal{J}}A_{\mathcal{J}}\Lambda_{\mathcal{J}}B^{-T}1_{J}} + \norma{\mu}_{R_{\mathcal{J}^{c}}^{\dagger}}^{2}.
\end{align*}

To obtain the second equality of the theorem, we apply the Woodbury matrix identity:
\begin{equation*}
Q_{\mathcal{J}} = (I_{m} + A_{\mathcal{J}}\Lambda_{\mathcal{J}}A_{\mathcal{J}}^{T})^{-1} = I_{m} - A_{\mathcal{J}}M_{\mathcal{J}}A_{\mathcal{J}}^{T},
\end{equation*}
where $M_{\mathcal{J}} := (\Lambda_{\mathcal{J}}^{-1} + A_{\mathcal{J}}^{T}A_{\mathcal{J}})^{-1}$. It can easily be shown that
\begin{align*}
Q_{\mathcal{J}}A_{\mathcal{J}}\Lambda_{\mathcal{J}} &= A_{\mathcal{J}}M_{\mathcal{J}},\\
(A_{\mathcal{J}}\Lambda_{\mathcal{J}})^{T}Q_{\mathcal{J}}A_{\mathcal{J}}\Lambda_{\mathcal{J}} &= \Lambda_{\mathcal{J}}^{-1}- M_{\mathcal{J}}.
\end{align*}
As such
\begin{equation*}
1_{J}^{T}B^{-1}\Lambda_{\mathcal{J}}B^{-T}1_{J} - (A_{\mathcal{J}}\Lambda_{\mathcal{J}}B^{-T}1_{J})^{T}Q_{\mathcal{J}}A_{\mathcal{J}}\Lambda_{\mathcal{J}}B^{-T}1_{J} = 1_{J}^{T}B^{-1}M_{\mathcal{J}}B^{-T}1_{J}.
\end{equation*}

Then
\begin{align*}
z_{\mathcal{J}}^{T}Q_{\mathcal{J}}z_{\mathcal{J}}& + \frac{(1 - 1_{J}^{T}B^{-1}(\mathcal{V}_{\mathcal{J}}^{T}\mu + \Lambda_{\mathcal{J}}A_{\mathcal{J}}^{T}Q_{\mathcal{J}}z_{\mathcal{J}}))^{2}}{ 1_{J}^{T}B^{-1}\Lambda_{\mathcal{J}}B^{-T}1_{J} - (A_{\mathcal{J}}\Lambda_{\mathcal{J}}B^{-T}1_{J})^{T}Q_{\mathcal{J}}A_{\mathcal{J}}\Lambda_{\mathcal{J}}B^{-T}1_{J}}   \\ 
  &= z_{\mathcal{J}}^{T}z_{\mathcal{J}} - z_{\mathcal{J}}^{T}A_{\mathcal{J}}M_{\mathcal{J}}A_{\mathcal{J}}^{T}z_{\mathcal{J}} + \frac{(1 - 1_{J}^{T}B^{-1}(\mathcal{V}_{\mathcal{J}}^{T}\mu + M_{\mathcal{J}}A_{\mathcal{J}}^{T}z_{\mathcal{J}}))^{2}}{1_{J}^{T}B^{-1}M_{\mathcal{J}}B^{-T}1_{J}}\\
&= z_{\mathcal{J}}^{T}z_{\mathcal{J}} - z_{\mathcal{J}}^{T}A_{\mathcal{J}}M_{\mathcal{J}}A_{\mathcal{J}}^{T}z_{\mathcal{J}} + \frac{(1 - 1_{J}^{T}B^{-1}M_{\mathcal{J}}\mathcal{V}_{\mathcal{J}}^{T}(A^{T}y + R_{\mathcal{J}}^{\dagger}\mu))^{2}}{1_{J}^{T}B^{-1}M_{\mathcal{J}}B^{-T}1_{J}}, 
\end{align*}
where, for the last line, we have used
\begin{align*}
\mathcal{V}_{\mathcal{J}}^{T}\mu + M_{\mathcal{J}}A_{\mathcal{J}}^{T}z_{\mathcal{J}} &=
\mathcal{V}_{\mathcal{J}}^{T}\mu + M_{\mathcal{J}}A_{\mathcal{J}}^{T}y - M_{\mathcal{J}}A_{\mathcal{J}}^{T}A_{\mathcal{J}}\mathcal{V}_{\mathcal{J}}^{T}\mu\\
&= \mathcal{V}_{\mathcal{J}}^{T}\mu + M_{\mathcal{J}}A_{\mathcal{J}}^{T}y - (I - M_{\mathcal{J}}\Lambda_{\mathcal{J}}^{-1})\mathcal{V}_{\mathcal{J}}^{T}\mu\\
&= M_{\mathcal{J}}\mathcal{V}_{\mathcal{J}}^{T}(A^{T}y + R_{\mathcal{J}}^{\dagger}\mu).
\end{align*}
\end{proof}

\noindent
\title{Proof of \textbf{Corollary \ref{cor:double optimal B condition}}}
\begin{proof}
\phantomsection\label{proof: double optimal B condition}
We define the following optimisation problem:
\begin{equation*}
    \min_{B \in \mathcal{B}}\norma{A\bar{u}_{\mathcal{J}, B}(0) - y}^{2} + \norma{\bar{u}_{\mathcal{J}, B}(0) - \mu}_{R}^{2} = \norma{\mu}_{R_{\mathcal{J}^{c}}}^{2} + \min_{B \in \mathcal{B}}\norma{A_{\mathcal{J}}B\frac{1}{J}1_{J} - y}^{2} + \norma{\mathcal{V}_{\mathcal{J}}B\frac{1}{J}1_{J}}_{R_{\mathcal{J}}}^{2},
\end{equation*}
where
\begin{equation*}
    \mathcal{B} := \{B \in \mathbb{R}^{J \times J} : 1 = 1_{J}^{T}B^{-1}   M_{\mathcal{J}}\mathcal{V}_{\mathcal{J}}^{T}(A^{T}y + R_{\mathcal{J}}^{\dagger}\mu)\}.
\end{equation*}
Defining $w := \frac{1}{J}B1_{J}$, $v := B^{-T}1_{J}$ and $\tilde{z} = M_{\mathcal{J}}\mathcal{V}_{\mathcal{J}}^{T}(A^{T}\tilde{y} + R_{\mathcal{J}}^{\dagger}\mu)$, we have the equivalent optimisation problem:
\begin{equation*}
    \min_{v, w} (\norma{A_{\mathcal{J}}w - y}^{2} + \norma{\mathcal{V}_{\mathcal{J}}w}_{R}^{2}) \mbox{ s.t. } v^{T}w = 1 \mbox{ \& } v^{T}\tilde{z} = 1.
\end{equation*}
We apply the method of Lagrange. First we define the Lagrangian: 
\begin{equation*}
    \mathcal{L}(v, w, \alpha, \beta) = \norma{A_{\mathcal{J}}w - y}^{2} + w^{T}\Lambda_{\mathcal{J}}w - \alpha (v^{T}w - 1) - \beta (\tilde{z}^{T}v - 1).
\end{equation*}
The first necessary condition for optimality is given by
\begin{align*}
    0 &= \frac{\partial \mathcal{L}}{\partial w }\\ 
    &= 2A_{\mathcal{J}}^{T}A_{\mathcal{J}}w - 2A_{\mathcal{J}}^{T}y + 2 \Lambda_{\mathcal{J}}w - \alpha v\\
    &= 2 M_{\mathcal{J}}^{-1}w - 2A_{\mathcal{J}}^{T}y - \alpha v,
\end{align*}
and therefore by
\begin{equation*}
    0 = w - M_{\mathcal{J}}A_{\mathcal{J}}^{T}y- \frac{\alpha}{2}M_{\mathcal{J}}v.
\end{equation*}
The second necessary condition for optimality is given by
\begin{align*}
0 &= \frac{\partial \mathcal{L}}{\partial v } \\
&= -\alpha w - \beta \tilde{z}.
\end{align*}
It then follows that 
\begin{equation*}
\alpha = \alpha(v^{T}w)= -\beta v^{T}\tilde{z} = -\beta,
\end{equation*}
and therefore
\begin{equation*}
0 = \alpha(w -\tilde{z}).
\end{equation*}
From the first necessary condition for optimality, we have
\begin{align*}
0 &= v^{T}(w - \tilde{z} - \frac{\alpha}{2}M_{\mathcal{J}}v)= 1 - 1 + \frac{\alpha}{2}v^{T}M_{\mathcal{J}}v.
\end{align*}
Since $v \neq 0$ and $M_{\mathcal{J}} \succ 0$, we therefore have that $\alpha = 0$. It then follows that
\begin{align*}
    w &= \tilde{z},\\
    \tilde{z}^{T}v &= 1,\\
    \alpha &= 0,\\
    \beta &= 0,
\end{align*}
which leads to
\begin{equation*}
    B1_{J} = JM_{\mathcal{J}}A_{\mathcal{J}}^{T}y.
\end{equation*}
\end{proof}

\noindent
\title{Proof of \textbf{Lemma} \ref{lemma:greedy}} 
\begin{proof}
\phantomsection\label{proof:greedy}
Recall that $x \in \mathbb{R}^{m}$, $P_{k} := [p_{1}, \ldots, p_{k}] \in \mathbb{R}^{m \times k}$, $\tilde{M}_{k} := (I_{k} + P_{k}^{T}P_{k})^{-1} \in \mathbb{R}^{k \times k}$ and $\Delta_{k + 1} := 1 + p_{k+1}^{T}p_{k+1} - p_{k+1}^{T}P_{k}\tilde{M}_{k}P_{k}^{T}p_{k+1}$. Then
\begin{align*}
    \tilde{M}_{k + 1} &= \begin{pmatrix} I_{k} + P_{k}^{T}P_{k} & P_{k}^{T}p_{k+1} \\ p_{k+1}^{T}P_{k} & 1 + p_{k+1}^{T}p_{k+1} \end{pmatrix}^{-1}\\
    &= \begin{pmatrix} \tilde{M}_{k}^{-1} & P_{k}^{T}p_{k+1} \\ p_{k+1}^{T}P_{k} & 1 + p_{k+1}^{T}p_{k+1} \end{pmatrix}^{-1}\\
    &= \frac{1}{\Delta_{k+1}}\begin{pmatrix} \Delta_{k+1}\tilde{M}_{k} + (M_{k}P_{k}^{T}p_{k+1})(\tilde{M}_{k}P_{k}^{T}p_{k+1})^{T} & -\tilde{M}_{k}P_{k}^{T}p_{k+1} \\ -p_{k+1}^{T}P_{k}\tilde{M}_{k} & 1 \end{pmatrix},
\end{align*}
where the last equality is due to block-wise matrix inversion. Then
\begin{align*}
&x^{T}P_{k+1}M_{k+1}P_{k+1}^{T}x \\
&= \frac{1}{\Delta_{k + 1}} \begin{pmatrix} P_{k}^{T}x \\ p_{k+1}^{T}x\end{pmatrix}^{T} \begin{pmatrix}\Delta_{k+1}\tilde{M}_{k} + (\tilde{M}_{k}P_{k}^{T}p_{k+1})(\tilde{M}_{k}P_{k}^{T}p_{k+1})^{T} & -\tilde{M}_{k}P_{k}^{T}p_{k+1} \\ -p_{k+1}^{T}P_{k}\tilde{M}_{k} & 1 \end{pmatrix}\begin{pmatrix} P_{k}^{T}x \\ p_{k+1}^{T}x\end{pmatrix}\\
&= x^{T}P_{k}\tilde{M}_{k}P_{k}^{T}x + \frac{1}{\Delta_{k+1}}(x^{T}P_{k}\tilde{M}_{k}P_{k}^{T}p_{k+1} - x^{T}p_{k+1})^{2}.
\end{align*}
\end{proof}

\noindent
\title{Proof of \textbf{Proposition \ref{proposition:optimal_case}}}
\begin{proof}
\phantomsection\label{proof:optimal_case}
Recall that $R = \mathcal{V}\Lambda\mathcal{V}^{T}$, A has a singular value decomposition $A = \mathcal{U}\Sigma\mathcal{V}^{T}$ and $\mathcal{V}_{\mathcal{J}} := [v_{\mathcal{J}_{1}}, \ldots, v_{\mathcal{J}_{J}}] \in \R^{n \times J}$. We further define $\mathcal{U}_{\mathcal{J}} := [u_{\mathcal{J}_{1}}, \ldots, u_{\mathcal{J}_{J}}] \in \R^{n \times J}$ and $\Sigma_{\mathcal{J}} := \mbox{diag}(\{\sigma_{\mathcal{J}_{1}}, \ldots, \sigma_{\mathcal{J}_{J}} \})$. Then
\begin{align*}
A_{\mathcal{J}} &= A\mathcal{V}_{\mathcal{J}} = \mathcal{U}_{\mathcal{J}}\Sigma_{\mathcal{J}},\\
z_{\mathcal{J}} &= y - AV_{\mathcal{J}}V_{\mathcal{J}}^{T}\mu = y - U_{\mathcal{J}}\Sigma_{\mathcal{J}}V_{\mathcal{J}}^{T}\mu,\\
A_{\mathcal{J}}^{T}z_{\mathcal{J}} &= \Sigma_{\mathcal{J}}\mathcal{U}_{\mathcal{J}}^{T}y - \Sigma_{\mathcal{J}}^{2}\mathcal{V}_{\mathcal{J}}^{T}\mu,\\
\mathcal{M}_{\mathcal{J}} &= (\Lambda_{\mathcal{J}}^{-1} + A_{\mathcal{J}}^{T}A_{\mathcal{J}})^{-1} = (\Lambda_{\mathcal{J}}^{-1} + \Sigma_{\mathcal{J}}^{2})^{-1}.
\end{align*}
It is then clear that
\begin{align}\label{eq:J objective specific A}
z_{\mathcal{J}}^{T}z_{\mathcal{J}} - z_{\mathcal{J}}^{T}A_{\mathcal{J}}\mathcal{M}_{\mathcal{J}}A_{\mathcal{J}}^{T}z_{\mathcal{J}} - \norma{\mu}^{2}_{R_{\mathcal{J}}^{\dagger}} =
&y^{T}y + \sum_{k \in \mathcal{J}} (- 2  \sigma_{k}(y^{T}u_{k})(\mu^{T}v_{k}) +  \sigma_{k}^{2}(\mu^{T}v_{k})^{2} \nonumber\\
&- \frac{\lambda_{k}(\sigma_{k}(y^{T}u_{k}) - \sigma_{k}^{2}(\mu^{T}v_{k}))^{2}}{1 + \lambda_{k}\sigma_{k}^{2}} - \frac{(\mu^{T}v_{k})^{2}}{\lambda_{k}}).
\end{align}

Examining Algorithm \ref{alg:selectJ}, we see that
\begin{align*}
    p_{k} &:= \lambda^{\frac{1}{2}}_{k}Av_{k}= \lambda^{\frac{1}{2}}_{k}\mathcal{U}\Sigma\mathcal{V}^{T}v_{k}=\lambda^{\frac{1}{2}}_{k}\sigma_{k}u_{k}.
\end{align*}
Since the $\{u_{k}\}_{k}$ are orthogonal, $\{ p_{k}\}_{k}$ are also orthogonal. Since $j^{*}\in \mathcal{J}_{k}^{c}$, at every stage of the algorithm, we have that 
\begin{align*}
P^{T}p_{j} &= 0,\\
z^{T}u_j &= (y - U_{\mathcal{J}}\Sigma_{\mathcal{J}}V_{\mathcal{J}}^{T}\mu)^{T}u_j = y^{T}u_j,\\
z^{T}Av_{j} &= \sigma_j (z^{T}u_{j}) = \sigma_{j}(y^{T}u_{j}),\\
\tilde{z}^{T}p_j &= (z - Av_{j}v_{j}^{T}\mu)^{T}(\lambda_j^{\frac12}\sigma_j u_j) = \lambda^{\frac{1}{2}}_{j}\sigma_{j}(y^{T}u_j - \sigma_{j}(\mu^{T}v_{j})).
\end{align*}
This yields
\begin{align}\label{eq:J objective specific A step}
-2(\mu^{T}v_{j})&(z^{T}Av_{j}) + (\mu^{T}v_{j})^{2}(\norma{Av_{j}}^{2} - \frac{1}{\lambda_{j}}) - \frac{(\tilde{z}^{T}P\tilde{M}P^{T}p_{j} - \tilde{z}^{T}p_j)^{2}}{1 + \norma{p_{j}}^{2} - p_{j}^{T}P\tilde{M}P^{T}p_{j}} = \nonumber\\
& -2\sigma_{j}(\mu^{T}v_{j})(y^{T}u_{j}) + (\mu^{T}v_{j})^{2}(\sigma_{j}^{2} - \frac{1}{\lambda_{j}}) - \frac{\lambda_{j}\sigma_{j}^{2}(y^{T}u_{j} - \sigma_{j}(\mu^{T}v_{j}))^{2}}{1 + \lambda_{j}\sigma_{j}^{2}}.
\end{align}
The RHS of this equation coincides with the summand in Equation \eqref{eq:J objective specific A}. As such, by choosing j at every stage of Algorithm \ref{alg:selectJ} to maximise the RHS of Equation \eqref{eq:J objective specific A step}, we end up maximising the objective function given by the LHS of Equation \eqref{eq:J objective specific A}.
\end{proof}
\bibliographystyle{authordate1}
\bibliography{refs}
\end{document}